\documentclass[12pt,oneside,english]{amsart}
\usepackage[T1]{fontenc}
\usepackage[latin9]{inputenc}
\usepackage[a4paper]{geometry}
\usepackage{amsthm}
\usepackage{amsbsy}
\usepackage{amssymb}
\usepackage{graphicx}
\usepackage{esint}

\makeatletter
\numberwithin{equation}{section}
\numberwithin{figure}{section}
\theoremstyle{plain}
\newtheorem{thm}{\protect\theoremname}
  \theoremstyle{plain}
  \newtheorem{cor}[thm]{\protect\corollaryname}
  \theoremstyle{definition}
  \newtheorem{defn}[thm]{\protect\definitionname}
  \theoremstyle{remark}
  \newtheorem{rem}[thm]{\protect\remarkname}
  \theoremstyle{plain}
  \newtheorem{lem}[thm]{\protect\lemmaname}

\makeatother

\usepackage{babel}
  \providecommand{\corollaryname}{Corollary}
  \providecommand{\definitionname}{Definition}
  \providecommand{\lemmaname}{Lemma}
  \providecommand{\remarkname}{Remark}
\providecommand{\theoremname}{Theorem}

\begin{document}

\title{Fractal weyl law for skew extensions of expanding maps}

\author{Arnoldi Jean-François}

\address{Institut Fourier, UMR 5582, 100 rue des Maths, BP74 38402 St Martin
d'Hères. }

\email{jean-francois@ujf-grenoble.fr}


\begin{abstract}
We consider compact Lie groups extensions of expanding maps of the
circle, essentially restricting to $\mbox{U}(1)$ and $\mbox{SU}(2)$
extensions. The central object of the paper is the associated Ruelle
transfer (or pull-back) operator $\hat{F}$. Harmonic analysis yields
a natural decomposition $\hat{F}=\oplus\hat{F}_{\alpha}$, where $\alpha$
indexes the irreducible representation spaces. Using Semiclassical
techniques we extend a previous result by Faure proving an asymptotic
spectral gap for the family $\left\{ \hat{F}_{\alpha}\right\} $ when
restricted to adapted spaces of distributions. Our main result is
a fractal Weyl upper bound for the number of eigenvalues (the Ruelle
resonances) of these operators out of some fixed disc centered on
$0$ in the complex plane.
\end{abstract}
\maketitle
\tableofcontents{}

\section{Introduction}

Partially hyperbolic dynamical systems form a class of chaotic models
of a more subtle type than the ones provided by the well understood
uniformly hyperbolic setting \cite{baladi_livre_00,brin-02}. The
difficulty in treating such systems is caused by the presence of a
neutral bundle that can drastically slow (or even prevent) the escape
towards statistical equilibrium. Compact Lie groups extension are
an attractive class of models of such type (the fibers of the neutral
bundle are homeomorphic to a given compact Lie group and the system
acts isometrically between fibers), where representation theory and
Lie groups techniques can help tackle issues such as mixing rates
or stable ergodicity. 

In this article we focus on skew extensions of expanding maps of the
circle. Put $S^{1}\equiv\mathbb{R}/\mathbb{Z}$ and let $E:S^{1}\rightarrow S^{1}$
be a $C^{\infty}-$expanding map with $k>1$ smooth inverse branches
$E_{\epsilon}^{-1}$ $\epsilon=0,...,k-1$. This map is automatically
topologically mixing \cite{katok_hasselblatt} and mixing w.r. to
a smooth invariant absolutely continious measure $\mu$ called the
SRB measure. Let $\mathbb{G}$ be a compact Lie group with normalized
Harr measure $m$. For any smooth map $\tau:S^{1}\rightarrow\mathbb{G}$
the skew extension $\widehat{E_{\tau}}:S^{1}\times\mathbb{G}\rightarrow S^{1}\times\mathbb{G}$
is defined by
\begin{equation}
\widehat{E_{\tau}}(x,g)=\left(E(x),\tau(x)g\right).\label{eq:extension}
\end{equation}
 The measure $\hat{\mu}_{\tau}:=\mu\times m$ is an invariant smooth
absolutely continuous measure for $\widehat{E_{\tau}}$. For any two
reasonably regular observables $\Psi,\Phi$ on $S^{1}\times\mathbb{G}$
the central object in ergodic theory is the correlation function,
defined for any $n\in\mathbb{N}^{*}$ by
\begin{equation}
C_{\Psi,\Phi}(n):=\left(\Psi\circ\widehat{E_{\tau}}^{n};\Phi\right)_{L^{2}\left(\hat{\mu}_{\tau}\right)}\label{eq:correlation-fct}
\end{equation}
which converges to $\int\bar{\Psi}d\hat{\mu}_{\tau}\int\Phi d\hat{\mu}_{\tau}$
as $n\rightarrow\infty$ iff $\widehat{E_{\tau}}$ is mixing. Fundamental
questions concern the rate of decay of $C_{\Psi,\Phi}(n)-\int\bar{\Psi}d\hat{\mu}_{\tau}\int\Phi d\hat{\mu}_{\tau}$
and are related to the spectral properties of the so-called \emph{Ruelle
transfer operator} 
\begin{equation}
\widehat{F}_{\tau}:\Psi\mapsto\Psi\circ\widehat{E_{\tau}},\;\Psi\in C^{\infty}\left(S^{1}\times\mathbb{G}\right)\label{eq:transfer_op}
\end{equation}
in adapted Banach spaces of distribution \cite{ruelle_75,ruelle_book}.
The first quantitative results in this context where obtained by Dolgopyat
\cite{dolgopyat_02} who showed an exponential decay rate (exponential
mixing) for generic maps $\tau$. On the other hand Naud \cite{Naud_08,Naud11},
in the analytic context showed that the rate of mixing cannot exceed
a certain bound related to the topological pressure of $-2\log\left|E'\right|$.
In the Abelian case $\mathbb{G}\equiv\mbox{U}(1)$ Faure \cite{fred_gap_09},
using semi-classical analysis showed that the transfer operator acting
in some Hilbert spaces of distributions generically exhibits an essential
spectral radius bounded by $1/\sqrt{E_{\min}}$; with $E_{\min}$
the minimal expansion rate of $E$. This results already deduced by
Tsujii \cite{Tsujii_07} in the setting of suspension semi-flows,
shows that up to terms of order $\rho^{n}$, $1>\rho>1/\sqrt{E_{\min}}$
the escape towards equilibrium is governed by a linear finite rank
operator (theorem 5 in \cite{fred_gap_09} and Eq.(1) in \cite{Tsujii_07}).
We will show that this result extends to the simplest non Abelian
compact lie group $\mathbb{G}\equiv\mbox{SU}(2)$ ( Theorem \ref{thm:gap}
and Corollary \ref{cor:Mixing}). 

As mentioned above to treat such models one uses harmonic analysis.
In particular the celebrated Peter-Weyl theorem \cite{taylor_tome1}
gives
\begin{equation}
L^{2}\left(\mathbb{G}\right)=\oplus_{\alpha}\mbox{dim}_{\mathbb{C}}\left(\mathcal{D}_{\alpha}\right)\mathcal{D}_{\alpha},\label{eq:peter_weyl}
\end{equation}
 where $\mathcal{D}_{\alpha}$ are finite dimensional irreducible
hermitian vector spaces of representation for $\mathbb{G}$. For $\mathbb{G}\equiv\mbox{U}(1)$
this is simply the Fourier decomposition of functions. The transfer
operator (\ref{eq:transfer_op}) extends to a continuous operator
on $L^{2}\left(S^{1}\times\mathbb{G}\right)=L^{2}\left(S^{1}\right)\otimes L^{2}\left(\mathbb{G}\right)$
and preserves the decomposition induced by (\ref{eq:peter_weyl}),
so that
\begin{equation}
\hat{F}_{\tau}=\oplus_{\alpha}\mbox{dim}_{\mathbb{C}}\left(\mathcal{D}_{\alpha}\right)\hat{F}_{\alpha}\label{eq:decomp}
\end{equation}
 with $\hat{F}_{\alpha}:=\hat{F}_{\tau}|_{L^{2}(S^{1})\otimes\mathcal{D}_{\alpha}}$
acting on smooth vector valued functions $\boldsymbol{\varphi}:S^{1}\rightarrow\mathcal{D}_{\alpha}$
as 
\begin{equation}
\left(\hat{F}_{\alpha}\boldsymbol{\varphi}\right)(x)=\hat{\tau}_{\alpha}(x)\boldsymbol{\varphi}(x),\label{eq:reduced_transfer_op}
\end{equation}
 with $\hat{\tau}_{\alpha}(x)$ the representation in $\mathcal{D}_{\alpha}$
of $\tau(x)\in\mathbb{G}$. In some standard Sobolev spaces of distributions
these operators have discrete spectrum (the Ruelle spectrum of resonances,
Theorem \ref{thm:discrete_spect}). For the trivial representation
corresponding to $\mbox{dim}_{\mathbb{C}}\left(\mathcal{D}_{\alpha}\right)=1$
and $\hat{\tau}_{\alpha}(x)\equiv\mbox{Id}$, the constant function
is an obvious eigenfunction of $\hat{F}_{\alpha}$ with eigenvalue
one%
\footnote{For other $\alpha$ they might not be any eigenvalues on the unit
circle. However if $\tau$ maps $S^{1}$ into a closed subgroup $H$
of $\mathbb{G}$ or whenever $\tau$ is co-homologous to such a map,
meaning that $\tau=\eta-\eta\circ E$ with $\eta:S^{1}\rightarrow H\subset\mathbb{G}$,
Then they will be some eigenvalues on the unit circle as $\widehat{E_{\tau}}$
is not topologically transitive in this case, hence not weakly-mixing
\cite{katok_hasselblatt}.%
}. 

For simplicity we shall restrict ourselves to the cases $\mathbb{G}\equiv\mbox{U}(1)$
and $\mathbb{G}\equiv\mbox{SU}(2)$, respectively the Abelian and
the simplest non-Abelian compact Lie groups. For $\mathbb{G}\equiv\mbox{U}(1)$,
$\alpha=\nu\in\mathbb{Z}$ and $\hat{\tau}_{\nu}(x)=e^{i\nu\Omega(x)};\,\Omega\in C^{\infty}\left(S^{1}\right)$.
For $\mathbb{G}\equiv\mbox{SU}(2)$, $\alpha=j\in\frac{1}{2}\mathbb{N}$
and $\mbox{dim}_{\mathbb{C}}\left(\mathcal{D}_{j}\right)=2j+1$. The
point we wish to make here is that the spectral study of the familly
$\left\{ \hat{F}_{\alpha}\right\} $ is a well-posed semi-classical
problem, as in \cite{fred_gap_09} when $\mathbb{G}\equiv\mbox{U}(1)$.
To see this in the non Abelian setting we will use Lie groups coherent
states theory \cite{perelomov1}. Doing so we will derive a Fractal
weyl asymptotic for the number of resonances outside a fixed spectral
radius (Theorem \ref{thm:weyl_law}) in the (semi-classical) limit
$\nu\mbox{ or }j\rightarrow\infty$. As in the previous papers by
Faure \cite{fred-roy-sjostrand-07,fred-RP-06,fred_gap_09}, the techniques
used are derived from the works of Sjöstrand \cite{sjoestrand_90}
Zworski-Lin-Guillope \cite{zworski_lin_guillope_02}, Sjöstrand-Zworski
\cite{sjoestrand_07}, in the context of chaotic scattering.

\section{Statement of the results}

The operators $\hat{F}_{\alpha}$ defined in (\ref{eq:reduced_transfer_op})
extend to the distribution spaces $\mathcal{D}'\left(S^{1}\right)\otimes\mathcal{D}_{\alpha}$
by setting, for any $\boldsymbol{\psi}\in\mathcal{D}'\left(S^{1}\right)\otimes\mathcal{D}_{\alpha}$
, $\boldsymbol{\varphi}\in C^{\infty}\left(S^{1}\right)\otimes\mathcal{D}_{\alpha}$:
$\left(\hat{F}_{\alpha}\boldsymbol{\psi}\right)\left(\bar{\boldsymbol{\varphi}}\right):=\boldsymbol{\psi}\left(\overline{\hat{F}_{\alpha}^{*}\boldsymbol{\varphi}}\right)$,
with the $L^{2}-$adjoint $\hat{F}_{\alpha}^{*}$ given by 
\begin{equation}
\left(\hat{F}_{\alpha}^{*}\boldsymbol{\varphi}\right)(x)=\sum_{y\in E^{-1}\left\{ x\right\} }\frac{1}{E'(y)}\hat{\tau}_{\alpha}(y)\boldsymbol{\varphi}(y).\label{eq:adjoint}
\end{equation}
 Recall that for $m\in\mathbb{R}$ the Sobolev spaces $H^{m}\left(S^{1}\right)\subset\mathcal{D}'\left(S^{1}\right)$
consists of distributions $\psi$ (or continious functions if $m>1/2$)
whose fourier series $\hat{\psi}(\xi)$ satisfy $\left\Vert \psi\right\Vert _{H^{m}}:=\sum_{\xi\in2\pi\mathbb{Z}}\left|\left\langle \xi\right\rangle ^{m}\hat{\psi}(\xi)\right|^{2}<\infty$,
with $\left\langle \xi\right\rangle :=\left(1+\xi^{2}\right)^{1/2}.$
It can equivalently be written \cite{taylor_tome1}:
\begin{equation}
H^{m}\left(S^{1}\right):=\left\langle \hat{\xi}\right\rangle ^{-m}\left(L^{2}(S^{1})\right),\;\hat{\xi}:=-i\frac{d}{dx}.\label{eq:H^m}
\end{equation}
 $\left\langle \hat{\xi}\right\rangle ^{m}$ is a typical representative
of the class of Pseudo-Differential-Operators (PDO) of order $m$
(cf section \ref{sec:PDOtheory} with $\hbar=1$).
\begin{thm}
\label{thm:discrete_spect}(Ruelle \cite{ruelle_86}, Faure \cite{fred_gap_09}).
Here $\mathbb{G}$ can be any compact Lie group and $\hat{F}_{\alpha}$
is defined by (\ref{eq:decomp}) and (\ref{eq:reduced_transfer_op}).
Then $\forall m,$ $\forall\alpha$, the operator $\hat{F}_{\alpha}$
acts in $H^{m}\left(S^{1}\right)\otimes\mathcal{D}_{\alpha}$ and
has discrete spectrum outside a disc of radius $r_{m}:=e_{\min}^{m}(k/e_{\min})^{1/2}$,
with $e_{\min}:=\min_{x}E'(x)>1$. The generalized eigenvalues outside
this disc, along with they respective eigenspaces, do not depend on
$m$ and define the Ruelle spectrum of resonances of $\hat{F}_{\alpha}$.
See figure \ref{fig:resonances}.
\end{thm}
\begin{figure}

\includegraphics[scale=0.35]{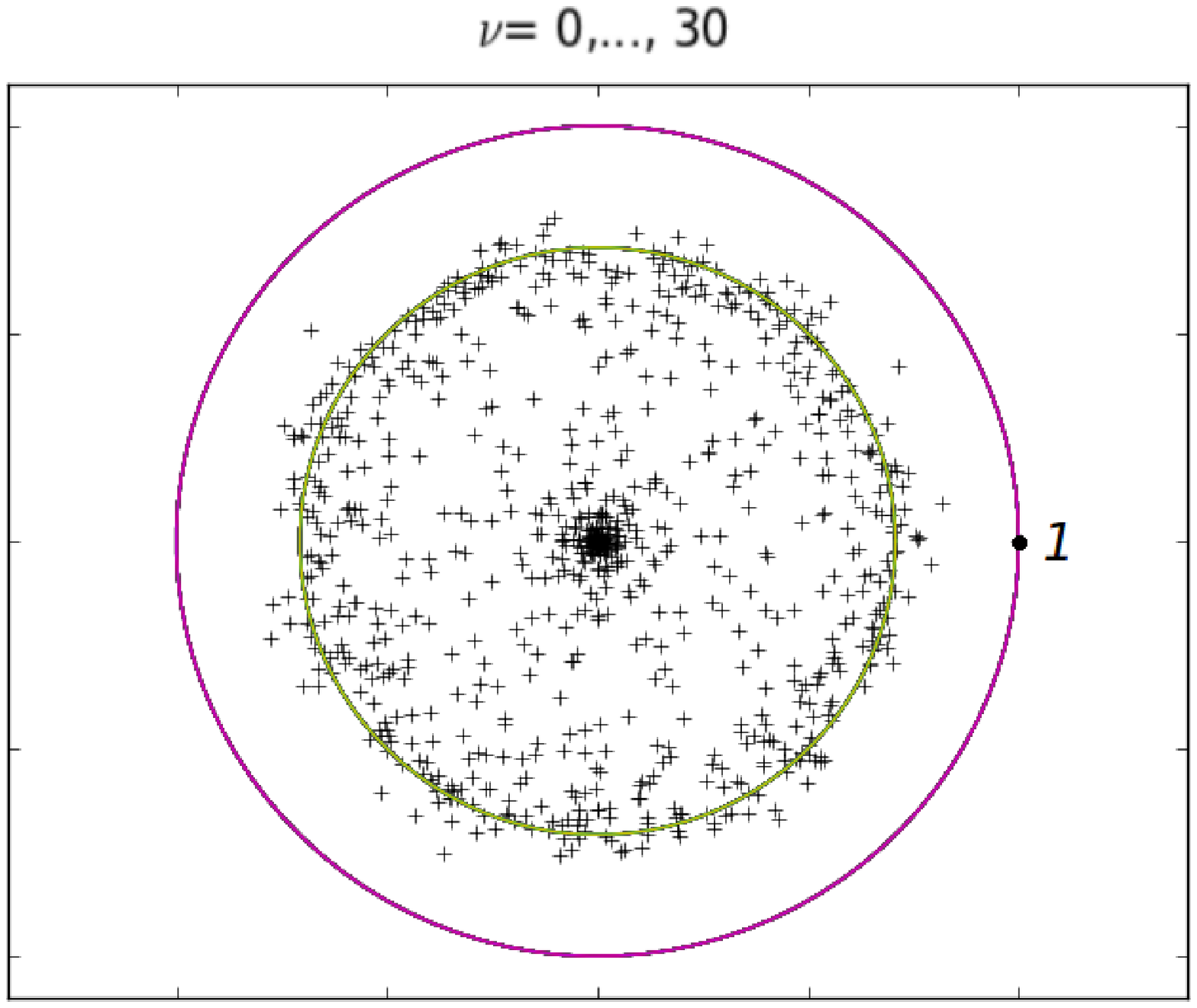}\includegraphics[scale=0.35]{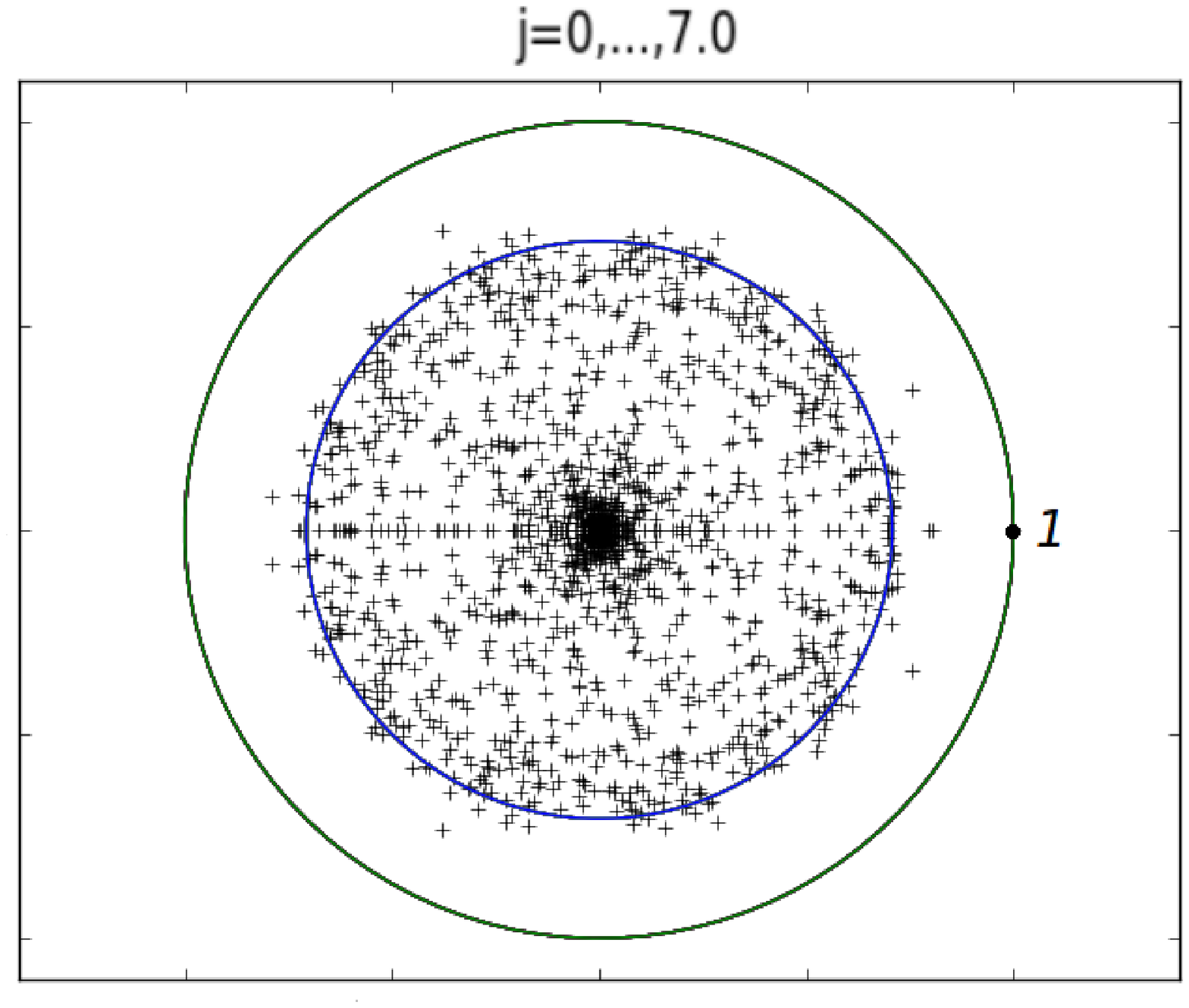}\caption{{\small \label{fig:resonances}Numerical computation of the superposition
of the resonance spectrum of $\hat{F}_{\alpha}$ (\ref{eq:reduced_transfer_op})
for to the first few values of $\alpha,$ for particular skew extension
of the linear map $E(x)=2x\mbox{ mod}1$. On the left $\tau(x)=\frac{1}{2\pi}\cos(2\pi x)$
seen as an element of the Abelian group $\mathbb{R}/\mathbb{Z}$.
On the right $\tau(x)=e^{i\cos(2\pi x)J_{3}}e^{i\theta J_{2}}e^{i\cos(2\pi x)J_{3}}\in\mbox{SU}(2)$
with $iJ_{l};l=1,2,3$ the generators of $\mathfrak{su}(2)$ and $\theta\neq0$
a fixed arbitrary value. In both pictures the inner circle corresponds
to the asymptotic gap of Theorem \ref{thm:gap} and the black dot
is the dominant simple eigenvalue $\lambda=1$.}}

\end{figure}

\begin{proof}
Let $\mathbb{G}$ be any compact Lie group and $\mathcal{D}_{\alpha}$
be some irreducible representation space for $\mathbb{G}$. Fix $m\in\mathbb{R}$
(we will write $H^{m}$, resp. $L^{2}$, for $H^{m}\left(S^{1}\right)$
and $L^{2}\left(S^{1}\right)$) and define $\mathcal{F}_{E}:\varphi\mapsto\varphi\circ E$
and $\left(\hat{\tau}\boldsymbol{\varphi}\right)(x):=\hat{\tau}_{\alpha}(x)\boldsymbol{\varphi}(x),$
so that $\hat{F}_{\alpha}=\hat{\tau}\left(\mathcal{F}_{E}\otimes\mathbb{I}\right)$.
Recall from \cite{fred_gap_09} that $\mathcal{F}_{E}$ restricts
to $\mathcal{F}_{E}:H^{m}\rightarrow H^{m}$ and has its essential
spectral radius bounded by $r_{m}$. Thus $\tilde{Q}_{m}:=\left\langle \hat{\xi}\right\rangle ^{m}\mathcal{F}_{E}\left\langle \hat{\xi}\right\rangle ^{-m}$
is $L^{2}-$continuous with the same essential spectral radius estimate.
Define $\hat{A}_{m}:=\left\langle \hat{\xi}\right\rangle ^{m}\otimes\mathbb{I}$
so that $\hat{A}_{m}^{-1}\left(L^{2}\otimes\mathcal{D}_{\alpha}\right)=H^{m}\otimes\mathcal{D}_{\alpha}$.
Consider $\hat{Q}_{m}=A_{m}\hat{F}_{\alpha}A_{m}^{-1}$ and $\hat{P}=\hat{Q}_{m}^{*}\hat{Q}_{m}$.
We see that $\hat{P}=\hat{A}_{m}^{-1}\hat{B}_{2m}\hat{A}_{m}^{-1}$
with $\hat{B}_{2m}:=\hat{F}_{\alpha}^{*}\hat{A}_{m}^{2}\hat{F}_{\alpha}=\left(\mathcal{F}_{E}^{*}\otimes\mathbb{I}\right)\hat{\tau}^{-1}\hat{A}_{m}^{2}\hat{\tau}\left(\mathcal{F}_{E}\otimes\mathbb{I}\right)$.
Commuting $\hat{A}_{m}^{2}$ and $\hat{\tau}$ we get that $\hat{B}_{2m}=\mathcal{F}_{E}^{*}\left\langle \hat{\xi}\right\rangle ^{2m}\mathcal{F}_{E}\otimes\mathbb{I}+\hat{C}_{2m-1}$,
with $\hat{C}_{2m-1}$ a matrix whose entries are PDOs of order $2m-1$
(cf section \ref{sec:PDOtheory} with $\hbar=1$). Therefore $\hat{P}=\tilde{Q}_{m}^{*}\tilde{Q}_{m}\otimes\mathbb{I}+\hat{C}_{-1}$
with $\hat{C}_{-1}$ a matrix whose entries are PDOs of order $-1$
hence compact. The independence in the value of $m$ for the discrete
eigenvalues is proven in \cite{fred-roy-sjostrand-07}, and is a consequence
of the fact that $H^{m'}$ is dense in $H^{m}$ for any $m'\geq m$.\end{proof}
\begin{thm}
\label{thm:gap}(Tsujii \cite{Tsujii_07} and Faure \cite{fred_gap_09}
for the Abelian case). Let $\mathbb{G}$ be either $\mbox{U}(1)$
or $\mbox{SU}(2).$ Recall that in these cases $\alpha$ denotes respectively
$\nu\in\mathbb{Z}$ or $j\in\frac{1}{2}\mathbb{N}$. For $m<0$ sufficiently
negative, if the map $\widehat{E_{\tau}}$ (\ref{eq:extension}) is
partially captive (definition \ref{partialy-captive}) then the spectral
radius $\hat{F}_{\alpha}:H^{m}\left(S^{1}\right)\otimes\mathcal{D}_{\alpha}\rightarrow H^{m}\left(S^{1}\right)\otimes\mathcal{D}_{\alpha}$
satisfies in the semiclassical limit $\alpha\rightarrow\infty$
\begin{equation}
r_{s}\left(\hat{F}_{\alpha}\right)\leq\frac{1}{\sqrt{E_{\min}}}+o(1),\label{eq:gap}
\end{equation}
 with $E_{\min}:=\lim_{n\rightarrow\infty}\left(\min_{x}\left(E^{n}\right)'(x)\right)^{1/n}>e_{\min}$
the minimal expansion rate of $E$. Also for any $\rho>\frac{1}{\sqrt{E_{\min}}}$
there exists $n_{0}$, $\alpha_{0}>0$, $m_{0}<0$ s.t. $\forall\left|\alpha\right|\geq\alpha_{0},\, m\leq m_{0}$
$\left\Vert \hat{F}_{\alpha}^{n_{0}}\right\Vert _{H_{1/\alpha}^{m}\otimes\mathcal{D}_{\alpha}}\leq\rho^{n_{0}}$,
where $\left\Vert \cdot\right\Vert _{H_{1/\alpha}^{m}}$ stands for
the Semiclassical Sobolev norm $\left\Vert \psi\right\Vert _{H_{1/\alpha}^{m}}:=\sum_{\xi\in2\pi\mathbb{Z}}\left|\left\langle \frac{1}{\alpha}\xi\right\rangle ^{m}\hat{\psi}(\xi)\right|^{2}$.
See figure \ref{fig:resonances}.
\end{thm}
As explained in detail in \cite{fred_gap_09}, from this and the decomposition
(\ref{eq:decomp}) one can deduce the mixing property of $\widehat{E_{\tau}}$.
We refer to the notations defined in the introduction. 
\begin{cor}
\label{cor:Mixing}Exponential Mixing of $\left(\widehat{E_{\tau}},S^{1}\times\mathbb{G}\right)$
for $\mathbb{G}\equiv\mbox{U}(1),\,\mbox{SU}(2)$. If the conclusion
of Theorem \ref{thm:gap} hold, than for any $\rho>\frac{1}{\sqrt{E_{\min}}}$
there exists a finite rank operator $\hat{k}$ such that, for any
observables \textup{$\Phi,\Psi\in C^{\infty}\left(S^{1}\times\mathbb{G}\right)$}
\[
\left(\widehat{F}_{\tau}^{n}\Psi;\Phi\right)_{L^{2}\left(\hat{\mu}_{\tau}\right)}=\left(\hat{k}^{n}\Psi;\Phi\right)_{L^{2}\left(\hat{\mu}_{\tau}\right)}+\mathcal{O}(\rho^{n}).
\]
Let $\Pi_{0}$ be the projector on $\mathcal{D}_{0}$. If $\lambda=1$
is the only eigenvalue of $\widehat{F}_{\tau}$ on the unit circle%
\footnote{which is always the case if $\tau$ is not a co-boundary: \cite{Tsujii_07}
Appendix A%
} than $\hat{k}$ admits a spectral decomposition $\hat{k}=\left|1\right\rangle \left\langle \mu\right|\otimes\Pi_{0}+\hat{r},$
$r_{s}\left(\hat{r}\right)<1$ with $\left|1\right\rangle \left\langle \mu\right|(\varphi)(x):=\left\langle \mu\right|\left.\varphi\right\rangle _{L^{2}(S^{1})}$
and $\mu$ the SRB measure of $E$. Thus
\[
C_{\Psi,\Phi}(n)-\int\bar{\Psi}d\hat{\mu}_{\tau}\int\Phi d\hat{\mu}_{\tau}=\left(\hat{r}^{n}\Psi;\Phi\right)_{L^{2}\left(\hat{\mu}_{\tau}\right)}+\mathcal{O}(\rho^{n}).
\]
 thus $\left(\widehat{E_{\tau}},S^{1}\times\mathbb{G}\right)$ is
exponentially mixing.\end{cor}
\begin{proof}
\cite{fred_gap_09} subsection 2.5.\end{proof}
\begin{defn}
For any $\lambda>0$ consider $\mathcal{O}_{\lambda}=\left\{ m\in\mathbb{R}\,|\, r_{m}\leq\lambda\right\} $.
The set of \emph{all resonances} of $\hat{F}_{\alpha}$ can be defined
as 
\[
\mbox{Res}\left(\hat{F}_{\alpha}\right):=\lim_{\lambda\rightarrow0}\bigcap_{m\in\mathcal{O}_{\lambda}}\mbox{spect}\left(\hat{F}_{\alpha}|_{H^{m}\left(S^{1}\right)\otimes\mathcal{D}_{\alpha}}\right).
\]

\end{defn}
In the spirit of \cite{sjoestrand_90} and \cite{zworski_lin_guillope_02}
we show:
\begin{thm}
\label{thm:weyl_law}Let $\mathbb{G}$ be either $\mbox{U}(1)$ or
$\mbox{SU}(2)$ as in Theorem \ref{thm:gap}. For any $\epsilon>0$
let $D_{\epsilon}^{\mathbb{C}}$ be the open disc in $\mathbb{C}$
of radius $\epsilon$. Then as $\alpha\rightarrow\infty$,
\begin{equation}
\sharp\left\{ \mbox{Res}\left(\hat{F}_{\alpha}\right)\cap\mathbb{C}\backslash D_{\epsilon}^{\mathbb{C}}\right\} =\mathcal{O}\left(\left|\alpha\right|^{\frac{1}{2}\dim\left(K_{\mathbb{G}}\right)+0}\right)\label{eq:fractal_weyl}
\end{equation}
 where $\dim\left(K_{\mathbb{G}}\right)$ is the upper Minkowski dimension
(definition \ref{The-upper-Minkowski}) of the trapped set $K_{\mathbb{G}}$
of the canonical map associated to $\hat{F}_{\alpha}$. For $\mathbb{G}=\mbox{U}(1)$
this map (\ref{eq:F_epsilon}) is defined on $T^{*}S^{1}$ and $K_{U(1)}$
is a compact subset of dimension between 1 and 2. For $\mathbb{G}=\mbox{SU}(2)$
this map (\ref{eq:F_epsilonNA}) is defined on $T^{*}S^{1}\times S^{2}$
and $K_{SU(2)}$ is a compact subset of dimension between 3 and 4.
This behaviour is tested numerically on figure \ref{fig:weyl-law}.\end{thm}
\begin{rem}
If we replace the arbitrary observables $\Phi,\Psi$ of Corollary
\ref{cor:Mixing} by functions $\Phi_{\alpha}(x,g),\Psi_{\alpha}(x,g)$
decomposing only on $C^{\infty}\otimes\mathcal{D}_{\alpha}$ (at $x$
fixed, eigenvalues of the Laplace operator on $\mathbb{G}$ of eigenvalue
$\lambda_{\alpha}$ \cite{taylor_tome1} p. 550) then, for any $\epsilon>0$,
there exists a finite rank operator $\hat{k}_{\alpha}$, s.t.
\[
\left(\widehat{F}_{\tau}^{n}\Psi_{\alpha};\Phi_{\alpha}\right)_{L^{2}\left(\hat{\mu}_{\tau}\right)}=\left(\hat{k}_{\alpha}^{n}\Psi_{\alpha};\Phi_{\alpha}\right)_{L^{2}\left(\hat{\mu}_{\tau}\right)}+\mathcal{O}(\epsilon^{n}).
\]
 with an estimation on the rank given by Theorem \ref{thm:weyl_law},
growing as $\alpha$ grows. 
\end{rem}
\begin{figure}
\includegraphics[scale=0.5]{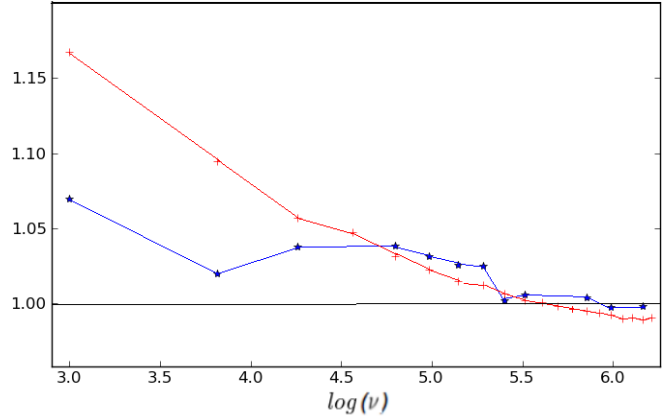}\caption{{\small \label{fig:weyl-law}Here we restrict our attention to the
Abelian skew extension of the linear map $E(x)=2x$ with $\tau(x)=\frac{1}{2\pi}\cos(2\pi x)\in\mathbb{R}/\mathbb{Z}$.
The +signs represent $\log(N_{\nu})/\log(\nu)$ with $N_{\nu}$ the
number of resonances of $\hat{F}_{\nu}$ larger in modulus than some
fixed $\epsilon>0$. The stars represent $1+\log Vol(K_{\nu})/\log\left(\nu\right)$
with $K_{\nu}$ a numerical approximation of the volume of a $\nu^{-\frac{1}{2}}-$neighbourhood
of the associated trapped set $K$ for $20\leq\nu\leq600$. If Theorem
\ref{thm:weyl_law} where sharp, then it would imply that both quantities
should converge at same speed to $\frac{1}{2}\dim K\sim1$ (see Lemma
\ref{lem:weyl_Q_m_mu}). Thus numerics suggest that it is the case.}}
\end{figure}

\section{$\hbar-$Pseudo differential theory\label{sec:PDOtheory}}

Before giving the proofs of Theorem \ref{thm:gap} and \ref{thm:weyl_law}
we first recall some basic facts from semiclassical analysis. This
will give us the opportunity to fix some notations but the reader
familiar with this theory can very well skip this section. We refer
to \cite{zworski-03,folland-88,martinez-01,dimassi-99}.

The symplectic structure of $\left(\mathbb{R}^{2d},\omega=dx\wedge d\xi\right)$
induces a Poisson algebra on the set of classical observables $C^{\infty}\left(\mathbb{R}^{2d};\mathbb{R}\right)$
if one defines the Poisson bracket of two such functions as $\left\{ f,g\right\} :=\omega\left(X_{f},X_{g}\right)$
with $X_{f}$ the Hamiltonian vector field generated by $f$. PDO
theory stems from an attempt to transpose such an algebra to the set
of formally self-adjoint operators acting on $L^{2}\left(\mathbb{R}^{d}\right)$,
the latter interpreted as the quantum Hilbert space. To some observable
$f$ we associate, $\forall\hbar>0$, and $\varphi\in\mathcal{S}$,
its Weyl quantization:
\begin{equation}
\mbox{Op}_{\hbar}^{w}(f)\varphi(x)=\frac{1}{\left(2\pi\hbar\right)^{d}}\int e^{\frac{i}{\hbar}\xi(x-y)}f\left(\frac{x-y}{2},\xi\right)\varphi(y)dyd\xi\label{eq:weyl_quant}
\end{equation}
 $f$ is than said to be the Weyl symbol of $\mbox{Op}_{\hbar}^{w}(f):\mathcal{S}\rightarrow\mathcal{S}'$.
This expression makes sense if $f$ is not real but if it is the operator
is formally self-adjoint. If one restricts to the following class
of symbols, given $m\in\mathbb{R}$ and $0\leq\mu<\frac{1}{2}$
\begin{equation}
S_{\mu}^{m}:=\left\{ a_{\hbar}\in C^{\infty}\left(\mathbb{R}^{2d}\right)|\,\left|\partial_{x}^{\alpha}\partial_{\xi}^{\beta}a_{\hbar}\right|\leq C_{\alpha\beta}\hbar^{-\mu(|\alpha|+|\beta|)}\left\langle \xi\right\rangle ^{m-|\beta|}\right\} \label{eq:S^m_nu}
\end{equation}
 then, for any $a_{\hbar}\in S_{\mu}^{m}$, $\mbox{Op}_{\hbar}^{w}(a_{\hbar})$
maps $\mathcal{S}$ to $\mathcal{S}$ and extends continuously to
a map $\mathcal{S}'\rightarrow\mathcal{S}'$. We write $OPS_{\mu}^{m}$
for the set of operators associated to $S_{\mu}^{m}$. This class
allows the construction of the Weyl quantization over smooth compact
manifolds, with (\ref{eq:weyl_quant}) taken in a local sense \cite{taylor_tome2}.
Symbols are then well defined as functions of the cotangent bundle.
In our case the manifold is the unit circle%
\footnote{In our context, since $S^{1}$ is an affine manifold it is not necessary
to restrict to such a class. But since our techniques allow it, and
since $S^{m}$ is standard we will not work on the most general class
allowed.%
} and its cotangent bundle the cylinder $T^{*}S^{1}.$ In the following
we drop the subscript $\hbar$ for symbols in $S_{\mu}^{m}$ even-though
they generally are one-parameter families of functions. 
\begin{lem}
\label{lem:L^2continuity}($L^{2}-$continuity theorem 5.1 in \cite{zworski-03}).
If $a\in S_{\mu}^{0}$ than for any $\hbar>0,$ $\mbox{\emph{Op}}_{\hbar}^{w}(a)$
extends to a continuous operator on $L^{2}$. As $\hbar$ goes to
zero, $\left\Vert \mbox{\emph{Op}}_{\hbar}^{w}(a)\right\Vert _{L^{2}\rightarrow L^{2}}\leq\sup|a|+\mathcal{O}(\hbar^{1-2\mu})$. 
\end{lem}
The set of $\hbar$-PDO defined through the weyl quantization is an
algebra and defines a star-algebra on the set of symbols, which coincides
at first order with the Poisson algebra of $C^{\infty}\left(\mathbb{R}^{2d}\right)$
induced by $\omega=dx\wedge d\xi$:
\begin{lem}
(Composition \cite{folland-88} p. 109). Let $a,\, b\in S_{\mu}^{m_{1}},\, S_{\mu}^{m_{2}}$.
Than $\mbox{\emph{Op}}_{\hbar}^{w}(a)\mbox{\emph{Op}}_{\hbar}^{w}(b)\in OPS_{\mu}^{m_{1}+m_{2}}$,
thus defining the star product $a\natural b$ such that $\mbox{\emph{Op}}_{\hbar}^{w}(a)\mbox{\emph{Op}}_{\hbar}^{w}(b)=\mbox{\emph{Op}}_{\hbar}^{w}(a\natural b)$.
$a\natural b=ab\mbox{ mod }\hbar^{1-2\mu}S_{\mu}^{m_{1}+m_{2}}$ and
Furthermore 
\[
\left[\mbox{\emph{Op}}_{\hbar}^{w}(a),\mbox{\emph{Op}}_{\hbar}^{w}(b)\right]=\frac{\hbar}{i}\mbox{\emph{Op}}_{\hbar}^{w}(\left\{ a,b\right\} )\mbox{ mod }\hbar^{2(1-2\mu)}S_{\mu}^{m_{1}+m_{2}-2}
\]
 
\end{lem}
The major consequence of this fact is that if $U_{t}$ is one parameter-group
of unitary operators satisfying the Shrödinger equation$-i\hbar\partial_{t}U_{t}:=\mbox{Op}_{\hbar}^{w}(H)U_{t},$
with $H$ a real bounded symbol, than the classical and {}``quantum''
dynamics are related at first order by the celebrated Egorov theorem:
\begin{lem}
(Egorov. section 9.2 in \cite{zworski-03}). For any $a\in S_{\mu}^{m}$,
$U_{-t}\mbox{\emph{Op}}_{\hbar}^{w}(a)U_{t}\in OPS_{\mu}^{m}$ and
its symbol is $a\circ e^{tX_{H}}\mbox{ mod }\hbar^{1-2\mu}S_{\mu}^{m-1}$,
with $e^{tX_{H}}$ the time-t flow generated by the Hamiltonian vector
field $X_{H}$. 
\end{lem}
The operators $U_{t}$ are a special kind of so called $\hbar-$Fourier-Integral-Operator
(FIO) \cite{martinez-01}. FIOs are always associated to a symplectic
map on the classical phase space. Another important example of FIOs
is given by pull-back operators $U_{\kappa}:\varphi\mapsto\varphi\circ\kappa$
with $\kappa$ a smooth diffeomorphism of $\mathbb{R}^{d}$. The above
Egorov theorem holds replacing the Hamiltonian flow by the canonical
lift $\widetilde{\kappa}$ of $\kappa$ on $\mathbb{R}^{2d}$ (seen
as the cotangent bundle of $\mathbb{R}^{d}$), 
\[
\widetilde{\kappa}:(x,\xi)\mapsto(\kappa^{-1}(x),^{t}D_{\kappa^{-1}(x)}\kappa\cdot\xi).
\]

\section{Canonical maps}

In this section we derive the canonical maps that play a central role
in our approach.

\subsection{The Abelian case}

When $\mathbb{G}\equiv\mbox{U}(1)$, the operator $\hat{F}_{\nu}$
reads $e^{i\nu\Omega}\mathcal{F}_{E}$ with $\mathcal{F}_{E}:\varphi\mapsto\varphi\circ E$
and $\Omega\in C^{\infty}(S^{1})$. $\hat{F}_{\nu}$ can be seen as
an $\hbar-$Fourier-Integral-Operator (FIO), with semiclassical parameter
$\nu^{-1}$. As explained in the previous section, the pull back operator
$\mathcal{F}_{E}$ is one of the simplest examples of an FIO and is
associated to the $k-$valued canonical lift $\widetilde{E}$ of $E$
on the cotangent space $T^{*}S^{1}$ (recall that $E_{\epsilon}^{-1}$
are the inverse branches of $E$):
\[
\widetilde{E}_{\epsilon}:(x,\xi)\mapsto(E_{\epsilon}^{-1}(x),E'\left(E_{\epsilon}^{-1}(x)\right)\xi)\;\epsilon=0,...,k-1.
\]
On the other hand the multiplication operator $e^{i\nu\Omega}$ ,$\nu\rightarrow\infty$
is also a very simple FIO and is associated to the time 1 flow generated
by the Hamiltonian $\Omega(x)$, $(x,\xi)\mapsto(x,\xi+\Omega'(x))$.
The canonical map associated to $\hat{F}_{\nu}$ is thus $k-$valued
and reads
\begin{equation}
F_{\epsilon}:(x,\xi)\mapsto(E_{\epsilon}^{-1}(x),E'\left(E_{\epsilon}^{-1}(x)\right)\xi+\Omega'\left(E_{\epsilon}^{-1}(x)\right))\;\epsilon=0,...,k-1.\label{eq:F_epsilon}
\end{equation}
 The intuitive idea behind these maps is that wave packets localized
both in direct and Fourier space near some point $(x,\nu\xi)=:(x,\xi_{\nu})$
will be transformed, up to negligible terms as $\nu\rightarrow\infty$,
in other wave packets localized near $F_{\epsilon}(x,\xi_{\nu}),$
$\epsilon=0,...,k-1.$ To the reader not familiar with semiclassical
analysis we recommend the reading of section 3.2 in \cite{fred_gap_09}
where this simple idea is explained in detail.

Using the fact that $E_{\epsilon}^{-1}\circ E=Id_{S^{1}}$ and (\ref{eq:adjoint})
the Egorov theorem of section \ref{sec:PDOtheory} can be quoted with
$\hat{F}_{\nu}$ in the role of the FIO as
\begin{lem}
\label{lem:Egorov-a}(Egorov in the Abelian setting). Let $\hbar=\nu^{-1};\nu>0$.
For any $a\in S_{\mu}^{m}\left(T^{*}S^{1}\right)$ any $\hbar>0,$
$\hat{F}_{\nu}^{*}Op_{\hbar}^{w}(a)\hat{F}_{\nu}\in OPS_{\mu}^{m}$
and its symbol reads
\begin{equation}
\sum_{\epsilon=0,...,k-1}\frac{a\circ F_{\epsilon}}{E'\circ E_{\epsilon}^{-1}}\mbox{ mod }\hbar^{1-2\mu}S_{\mu}^{m-1}.\label{eq:egorov_a}
\end{equation}
 
\end{lem}

\subsection{The simplest non-Abelian case}

When $\mathbb{G}\equiv\mbox{SU}(2)$, the operator $\hat{F}_{j}$
reads $\hat{\tau}_{j}\left(\mathcal{F}_{E}\otimes\mathbb{I}_{\mathcal{D}_{j}}\right)$.
We thus need to understand how the unitary operator $\hat{\tau}_{j}$
defined by $\left(\hat{\tau}_{j}\boldsymbol{\varphi}\right)(x):=\hat{\tau}_{j}(x)\boldsymbol{\varphi}(x)$
could be an FIO and on which space its canonical map would act.In
other words we have to define wave packets on $C^{\infty}(S^{1})\otimes\mathcal{D}_{j}$
and exhibit the action of $\hat{\tau}_{j}$ upon them. For this purpose
we shall use coherent states theory for compact Lie groups \cite{perelomov1}.

The Lie algebra of $\mbox{SU}(2)$ is a 3 dimensional real vector
space spanned by three generators $iJ_{l};l=1,2,3.$ The representations
spaces $\mathcal{D}_{j}$ are eigenspaces of the Casimir operator
$\mbox{J}^{2}:=\sum_{l=1}^{3}J_{l}^{2}$ with eigenvalue $j(j+1)$,
and an o.n. basis in $\mathcal{D}_{j}$ is given by the eigenvectors
of $J_{3}$. In particular if one defines $J_{\pm}=J_{1}\pm iJ_{2}$
there exists a unique normalized vector $\left|0\right\rangle \in\mathcal{D}_{j}$
such that $J_{3}\left|0\right\rangle =-j\left|0\right\rangle $ and
$J_{-}\left|0\right\rangle =0$, called the maximal weight vector
and $\mathcal{D}_{j}=\mbox{span}\left\{ J_{+}^{k}\left|0\right\rangle ,k=0,...,2j\right\} $.
Let $\pi:\mathcal{D}_{j}\rightarrow\mathbb{P}\left(\mathcal{D}_{j}\right)$
be the canonical mapping on the projective space. Define the following
subset of $\mathbb{P}\left(\mathcal{D}_{j}\right)$ 
\begin{equation}
X_{j}:=\left\{ \pi\left(\hat{g}_{j}\left|0\right\rangle \right);\, g\in\mathbb{G}\right\} ,\label{eq:X_j}
\end{equation}
 i.e. the set of complex lines trough $0$ in $\mathcal{D}_{j}$ spanned
by the orbit of $\left|0\right\rangle $. The isotropy subgroup of
$\pi\left(\left|0\right\rangle \right)$ is the set of elements $\left\{ e^{i\theta J_{3}}\right\} $,
a $\mbox{U}(1)$ subgroup of $\mathbb{G}$. Thus $X_{j}$ is isomorphic
to the coset space $\mbox{SU}(2)/\mbox{U}(1)\cong S^{2}$. The isomorphism
$\phi:X_{j}\rightarrow S^{2}$ can be fixed once and for all by setting
$\phi(\pi\left(\hat{g}_{j}\left|0\right\rangle \right))=R_{g}\mathbf{n}_{0}$
with $\mathbf{n}_{0}$ the south pole of the sphere and $R_{g}\in\mbox{SO}(3)$
the rotation corresponding to the adjoint representation of $g$ in
$\mathbb{R}^{3}$. A point $\mathbf{n}$ on the sphere is thus associated
to an orthogonal projector $\left|\mathbf{n}\right\rangle \left\langle \mathbf{n}\right|$
on $\mathcal{D}_{j}$ with $\left|\mathbf{n}\right\rangle $ being
any $\hat{g}_{j}\left|0\right\rangle $ s.t. $[g]\equiv\mathbf{n}$.
This mapping is called the \emph{quantization of the sphere} and the
vectors $\left|\mathbf{n}\right\rangle $ the \emph{coherent states}.
Furthermore, since $\int_{S^{2}}\left|\mathbf{n}\right\rangle \left\langle \mathbf{n}\right|d\mathbf{n}$
commutes with all elements of $\mathbb{G}$, by Shur's lemma it is
a multiple of the identity $\mathbb{I}_{\mathcal{D}_{j}}$. An algebraic
calculation gives (\cite{perelomov1} p. 63)
\begin{equation}
\frac{\dim\mathcal{D}_{j}}{4\pi}\int_{S^{2}}\left|\mathbf{n}\right\rangle \left\langle \mathbf{n}\right|d\mathbf{n}=\mathbb{I}_{\mathcal{D}_{j}}.\label{eq:completude}
\end{equation}
Also one can show that the coherent sates are localized, as $j\rightarrow\infty$,
on the sphere: $\left|\left\langle \mathbf{n'}\right|\left.\mathbf{n}\right\rangle \right|^{2}=\left|\frac{1+\mathbf{n}'\cdot\mathbf{n}}{2}\right|^{2j}$.
A natural way to quantize smooth observables is to put 
\begin{equation}
\mbox{Op}_{j}^{AW}:\begin{array}{ccc}
C^{\infty}(S^{2}) & \rightarrow & \mbox{End}(\mathcal{D}_{j})\\
a(\mathbf{n}) & \mapsto & \int_{S^{2}}a(\mathbf{n})\left|\mathbf{n}\right\rangle \left\langle \mathbf{n}\right|d\mathbf{n}_{j}
\end{array}\; d\mathbf{n}_{j}:=\frac{\dim\mathcal{D}_{j}}{4\pi}.\label{eq:anti-wick}
\end{equation}
 This mapping is surjective. In fact, because we have chosen a maximal
weigh vector as a starting point, $X_{j}$ has a Kaehlerian structure
inherited from the projective space (\cite{sternberg} p. 168). The
symplectic form on $X_{j}$ reads $\omega_{j}=j\omega_{S^{2}}$ with
$\omega_{S^{2}}$ the canonical symplectic form on $S^{2}$ and (\ref{eq:anti-wick})
corresponds to the geometric quantization of $\left(S^{2},\omega_{S^{2}}\right)$
with $\mathcal{D}_{j}$ as the quantum Hilbert space \cite{woodhouse2};
the following is a special case of a more general result:
\begin{thm}
\label{thm:(cahen-laurent-charles)}(Cahen-Gutt-Rawnsley \cite{Cahen_Gutt_Rawnsley91}).
For any $a,\, b\in C^{\infty}\left(S^{1}\right)$ 
\[
\mbox{\emph{Op}}_{j}^{AW}(a)\mbox{\emph{Op}}_{j}^{AW}(b)=\mbox{\emph{Op}}_{j}^{AW}(a\natural b),
\]
 with $a\natural b=ab+O(j^{-1})$. Furthermore 
\[
-\frac{i}{j}\left[\mbox{\emph{Op}}_{j}^{AW}(a),\mbox{\emph{Op}}_{j}^{AW}(b)\right]=\mbox{\emph{Op}}_{j}^{AW}(\left\{ a,b\right\} )+\mathcal{O}_{End(\mathcal{D}_{j})}(j^{-1}).
\]
\end{thm}
\begin{rem}
\label{note1}The representations of group elements are natural FIOs
in this setting as one can check readily that, for any observable
$a$, and $g\in\mathbb{G}$, 
\[
\hat{g}_{j}^{-1}\mbox{Op}_{j}^{AW}(a)\hat{g}_{j}=\mbox{Op}_{j}^{AW}(a\circ R_{g}).
\]
 Since $\hat{g}_{j}=e^{iju\cdot\frac{J}{j}}$ for some $u\in\mathbb{R}^{3}$
this means that $u\cdot\frac{J}{j}=\mbox{Op}_{j}^{AW}(u\cdot\mathbf{n}+\mathcal{O}(j^{-1}))$
so that the Hamiltonian time-1 flow on $S^{2}$ coincides with the
rotation $R_{g}\in\mbox{SO}(3)$. 
\end{rem}
We can now define a quantization of the full symplectic phase space
\begin{equation}
\left(T^{*}S^{1}\times S^{2};dx\wedge d\xi+\omega_{S^{2}}\right)\label{eq:full_phase_space}
\end{equation}
To any classical observable $a\in C^{\infty}\left(T^{*}S^{1}\times S^{2}\right)$
we associate an operator $\mbox{Op}_{j}\left(a\right):\mathcal{S}(S^{1})\otimes\mathcal{D}_{j}\rightarrow\mathcal{S}'(S^{1})\otimes\mathcal{D}_{j}$
defined by 
\begin{equation}
\mbox{Op}_{j}\left(a\right):=\mbox{Op}_{j}^{AW}\circ\mbox{Op}_{j^{-1}}^{w}\left(a\right)\label{eq:full_quantization}
\end{equation}
 with $\mbox{Op}_{\hbar}^{w};\,\hbar>0$ the usual Weyl quantization
(\ref{eq:weyl_quant}) defined on the circle. By composition the quantization
(\ref{eq:full_quantization}) obeys the correspondence principle:
The composition of two such PDO is a PDO whose principal term is the
product of the symbols and the principal term of their commutator
is $(-i/j)$ times the Poisson bracket of the symbols. Since $\mbox{SU}(2)$
is simply connected we can write $\tau(x)$ as $e^{i\Omega(x)\cdot J}$
with $\Omega\in C^{\infty}(S^{1};\mathbb{R}^{3})$ a smooth vector
valued function. Using remark \ref{note1}, we have that 
\[
\hat{\tau}_{j}=\exp\left(ij\mbox{Op}_{j}\left(a\right)\right);\,\mbox{with }a(x,\mathbf{n})=\Omega(x)\cdot\mathbf{n}+\mathcal{O}(j^{-1}).
\]
This is precisely the formal expression of an FIO associated to the
time-$1$ flow generated by the Hamiltonian vector field associated
to $\Omega(x)\cdot\mathbf{n}$:
\[
\dot{x}=0;\,\dot{\xi}=-\Omega'(x)\cdot\mathbf{n};\,\dot{\mathbf{n}}=\Omega(x)\wedge\mathbf{n}
\]
 After integration the time-1 flow reads: 
\[
(x,\xi,\mathbf{n})\mapsto(x,\,\xi+\mathbf{n}\cdot H(x,\mathbf{n}),\, R_{\tau(x)}\mathbf{n}),
\]
 with $H(x,\mathbf{n}):=\mathbf{n}\cdot\int_{0}^{1}\tilde{R}_{t\Omega(x)}\Omega'(x)dt$,
and $\tilde{R}_{u}$ the rotation in $\mathbb{R}^{3}$ of axis $u$.
Coming back to the operator $\hat{F}_{j}=\hat{\tau}_{j}\left(\mathcal{F}_{E}\otimes\mathbb{I}_{\mathcal{D}_{j}}\right)$
we see now that it is indeed an FIO and by composition its k-valued
canonical map reads, with $\epsilon=0,...,k-1:$ 
\begin{equation}
F_{\epsilon}:\left(x,\xi,\mathbf{n}\right)\mapsto\left(E_{\epsilon}^{-1}(x),E'\left(E_{\epsilon}^{-1}(x)\right)\xi+H(E_{\epsilon}^{-1}(x);\mathbf{n}),R_{\tau\left(E_{\epsilon}^{-1}(x)\right)}\mathbf{n}\right).\label{eq:F_epsilonNA}
\end{equation}

Using the Campbell-Hausdorff formula (\cite{taylor_tome1} p. 541)
one can show that%
\footnote{This result is not surprising if one considers the transport by $\hat{\tau}_{j}$
of generalized wave packets $\boldsymbol{\varphi}_{x,\xi,\mathbf{n}}:=\varphi_{x,\xi}\otimes\left|\mathbf{n}\right\rangle $
with $\varphi_{x,\xi}$ a Gaussian wave-packet of width $j^{-1/2}$
as defined in \cite{fred_gap_09} section 3.2. By the localization
property of both the coherent states and the Gaussian wave packets
\[
\left(\boldsymbol{\varphi}_{y,\eta,\mathbf{n}'}|\hat{\tau}_{j}\boldsymbol{\varphi}_{x,\xi,\mathbf{n}}\right)_{L^{2}\left(S^{1}\right)\otimes\mathcal{D}_{j}}:=\int_{S^{1}}\overline{\varphi_{y,\eta}(z)}\varphi_{x,\xi}(z)\left\langle \mathbf{n}'|\hat{\tau}_{j}(z)\mathbf{n}\right\rangle dz
\]
 will be negligible as $j$ grows if $y\neq x$ and if $y=x$ negligible
if $\mathbf{n}'\neq R_{\tau(x)}\mathbf{n}$. On the other hand, for
some $c>0$
\[
\left|\left(\boldsymbol{\varphi}_{x,\eta,R_{\tau(x)}\mathbf{n}}|\hat{\tau}_{j}\boldsymbol{\varphi}_{x,\xi,\mathbf{n}}\right)\right|:=\left|\int_{S^{1}}e^{-j(x-z)^{2}}e^{ij(\xi-\eta)}\left\langle \mathbf{n}|\hat{\tau}_{j}(x)^{-1}\hat{\tau}_{j}(z)\mathbf{n}\right\rangle dz\right|+\mathcal{O}(e^{-cj}).
\]
 $\left\langle \mathbf{n}|\hat{\tau}_{j}(x)^{-1}\hat{\tau}_{j}(z)\mathbf{n}\right\rangle $
is maximal and equal to one for $z=x$. If we write this term as $\rho(z)e^{-ij\psi(z)}$
then the stationary phase theorem states that the above expression
will be negligible whenever $\xi-\eta-\frac{i}{j}\left\langle \mathbf{n}|\hat{\tau}_{j}(x)^{-1}\frac{d}{dx}\hat{\tau}_{j}(x)\mathbf{n}\right\rangle \neq0$. %
} 
\[
H(x,\mathbf{n})=-\frac{i}{j}\left\langle \mathbf{n}\right|\hat{\tau}_{j}^{-1}\frac{d\hat{\tau}_{j}}{dx}(x)\left|\mathbf{n}\right\rangle ,
\]
 the Wick symbol \cite{Cahen_Gutt_Rawnsley91} of $-ij^{-1}\left(\hat{\tau}_{j}^{-1}\frac{d}{dx}\hat{\tau}_{j}\right)(x)\in\mathfrak{su}(2)$.

Set $\hbar\equiv j^{-1},\, j>0$ to define the class, as (\ref{eq:S^m_nu}),
given $m\in\mathbb{R}$ and $0\leq\mu<\frac{1}{2}$: 
\begin{equation}
S_{\mu}^{m}\left(T^{*}S^{1}\times S^{2}\right):=\left\{ a_{\hbar}\in C^{\infty}\,;\,|\partial_{x}^{\alpha}\partial_{\xi}^{\beta}\partial_{\mathbf{n}}^{\gamma}a_{\hbar}|\leq C_{\alpha\beta}\hbar^{-\mu(\alpha+\beta+|\gamma|)}\left\langle \xi\right\rangle ^{m-|\beta|}\right\} .\label{eq:S^m_mu(T*S_times_S^2)}
\end{equation}
In this setting the Egorov theorem can then be stated as 
\begin{lem}
\label{lem:Egorov-na}(Egorov in the non-Abelian setting). Let $a\in S_{\mu}^{m}\left(T^{*}S^{1}\times S^{2}\right)$,
then $\hat{F}_{j}^{*}Op_{j}\left(a\right)\hat{F}_{j}\in OPS_{\mu}^{m}$
and its principal symbol reads
\begin{equation}
\sum_{\epsilon=0,...,k-1}\frac{a\circ F_{\epsilon}}{E'\circ E_{\epsilon}^{-1}}\mbox{ mod }\hbar^{1-2\mu}S_{\mu}^{m-1}.\label{eq:egorov_na}
\end{equation}
 
\end{lem}

\section{Dynamics on Phase space}

In this section we derive some elementary facts about the canonical
maps (\ref{eq:F_epsilon}) and (\ref{eq:F_epsilonNA}) associated
respectively to $\left\{ \hat{F}_{\nu}\right\} _{\nu\in\mathbb{Z}}$
and $\left\{ \hat{F}_{j}\right\} _{j\in\frac{1}{2}\mathbb{N}}.$ We
introduce the compact trapped sets of Theorem \ref{thm:weyl_law}
and give the hypothesis on which Theorem \ref{thm:gap} is based.

\subsection{Time$-n$ dynamics\label{sub:Time-n-dynamics} }

Let $\mathcal{A}:=\{0,...,k-1\}$ be the alphabet and 
\[
\mathcal{A}^{n}:=\left\{ \epsilon=\epsilon_{1}\epsilon_{2}...\epsilon_{2};\epsilon_{i}\in\mathcal{A}\right\} ,
\]
 the set of $k^{n}$ words on length $n>0$ written with $\mathcal{A}$.
For any $\epsilon\in\mathcal{A}^{n}$ define $E_{\epsilon}^{-n}:=E_{\epsilon_{n}}^{-1}\circ...\circ E_{\epsilon_{1}}^{-1}$
and put $x_{\epsilon}:=E_{\epsilon}^{-n}(x)$. The expansion rate
of this trajectory is then 
\[
E'_{\epsilon}(x):=\left(E^{n}\right)'(x_{\epsilon})=\prod_{j=1}^{n}E'(x_{\epsilon|_{j}})
\]
 with $\epsilon|_{j}:=\epsilon_{1}...\epsilon_{j}$ the troncation
at the $j-$th letter of the word $\epsilon$. We put $\forall\epsilon\in\mathcal{A}^{n}$:
\[
\xi_{x,\epsilon}:=E'_{\epsilon}(x)\{\xi-S_{\epsilon}^{U(1)}(x)\}
\]
 with
\begin{equation}
S_{\epsilon}^{U(1)}(x):=-\sum_{j=1}^{n}E_{\epsilon|_{j}}'(x)^{-1}\cdot\Omega'\left(x_{\epsilon|_{j}}\right),\label{eq:S_U(1)}
\end{equation}
and $\Omega$ as in (\ref{eq:F_epsilon}). In the same manner, if
\[
\mathbf{n}_{x,\epsilon}:=R_{\tau(x_{\epsilon})}\circ R_{\tau(x_{\epsilon|_{n-1}})}\circ...\circ R_{\tau(x_{\epsilon_{1}})}\mathbf{n}
\]
 define 
\[
\xi_{x,\mathbf{n},\epsilon}:=E'_{\epsilon}(x)\{\xi-S_{\epsilon}^{SU(2)}(x,\mathbf{n})\}
\]
 with
\begin{equation}
S_{\epsilon}^{SU(2)}(x,\mathbf{n}):=-\sum_{j=1}^{n}E_{\epsilon|_{j}}'(x)^{-1}\cdot H_{\dagger}\left(x_{\epsilon|_{j}},\mathbf{n}_{x,\epsilon|_{j}}\right),\label{eq:S_SU(2)}
\end{equation}
 and $H_{\dagger}(x,\mathbf{n}):=H(x,R_{\tau(x)}^{-1}\mathbf{n})$,
$H$ as in (\ref{eq:F_epsilonNA}). With these notations the time-$n$
dynamics read
\begin{equation}
F_{\epsilon}^{n}(x,\xi)=(x_{\epsilon},\xi_{x,\epsilon})\mbox{ and }F_{\epsilon}^{n}(x,\xi,\mathbf{n})=(x_{\epsilon},\xi_{x,\mathbf{n},\epsilon},\mathbf{n}_{x,\epsilon});\,\epsilon\in\mathcal{A}^{n}\label{eq:F^n}
\end{equation}
 for respectively the Abelian and non-Abelian case. The inverse maps
$F^{-n}$ are single valued and read, 
\begin{equation}
F^{-n}(x,\xi)=(\; E^{n}x,\left(E^{n}\right)'(x)^{-1}\{\xi-\sum_{j=0}^{n-1}\left(E^{j}\right)'(x)\cdot\Omega'\left(E^{j}x\right)\}\;)\label{eq:F^-n_abel}
\end{equation}
in the Abelian setting, and if $\mathbf{n}_{x}^{(j)}:=R_{\tau(E^{j-1}x)}^{-1}\circ...\circ R_{\tau(x)}^{-1}\mathbf{n}$,
in the non-Abelian setting: 
\begin{equation}
F^{-n}(x,\xi,\mathbf{n})=(\; E^{n}x,\;\left(E^{n}\right)'(x)^{-1}\{\xi-\sum_{j=0}^{n-1}\left(E^{j}\right)'(x)\cdot H_{\dagger}(E^{j}x,\mathbf{n}_{x}^{(j)})\},\;\mathbf{n}_{x}^{(n)}\;)\label{eq:F^-n_non_abel}
\end{equation}

\subsection{The trapped sets }

We refer to the notations defined in the previous subsection. A fundamental
feature of the classical dynamics is that:
\begin{lem}
\label{lem:fuite,}For any $1<\kappa<e_{\min}$, there exists $R>0$,
s.t. for any $\epsilon\in\mathcal{A}$, $\left|\xi\right|\geq R\Rightarrow\left|\xi_{x,\epsilon}\right|\geq\kappa\left|\xi\right|$
and $\left|\xi_{x,\mathbf{n},\epsilon}\right|\geq\kappa\left|\xi\right|$. \end{lem}
\begin{proof}
If $\xi>0$, than $\xi_{x,\epsilon}$ (resp. $\xi_{x,\mathbf{n},\epsilon}$)
will be larger than $\kappa\xi$ for some $1<\kappa<e_{\min}$ iff
$\kappa\xi\geq e_{\min}\xi-C_{-}^{\mathbb{G}}\Leftrightarrow\xi\geq C_{-}^{\mathbb{G}}/(e_{\min}-\kappa)$,
with $C_{-}^{U(1)}=\left|\min_{x}\Omega'(x)\right|$ and $C_{-}^{SU(2)}=\left|\min_{x,\mathbf{n}}H(x,\mathbf{n})\right|$.
On the other hand if $\xi<0$ then $\xi_{x,\epsilon}$ (resp. $\xi_{x,\mathbf{n},\epsilon}$)
will be smaller than $\kappa\xi$ iff $\kappa\xi\leq e_{\min}\xi+C_{+}^{\mathbb{G}}\Leftrightarrow\xi\leq-C_{+}^{\mathbb{G}}/(e_{\min}-\kappa)$,
with $C_{+}^{U(1)}=\left|\max_{x}\Omega'(x)\right|$ and $C_{+}^{SU(2)}=\left|\max_{x,\mathbf{n}}H(x,\mathbf{n})\right|$.
So with $R:=\max_{\mathbb{G}}\max\left\{ C_{-}^{\mathbb{G}},C_{+}^{\mathbb{G}}\right\} /(e_{\min}-\kappa)$
the lemma holds.
\end{proof}
As a consequence, for both maps (\ref{eq:F_epsilon}) and (\ref{eq:F_epsilonNA})
there exists a non-empty compact set of points from which some trajectories
do not escape as $n\rightarrow\infty$. Indeed, taking $R>0$ large
enough and $Z_{U(1)}:=S^{1}\times[-R;R]$, $Z_{SU(2)}:=S^{1}\times[-R;R]\times S^{2}$,
$K_{\mathbb{G}}$ can be defined as the limit of a sequence of nested
non-empty compacts sets (see fig. \ref{fig:K_G}):
\begin{equation}
K_{\mathbb{G}}:=\bigcap_{n\geq0}F^{-n}\left(Z_{\mathbb{G}}\right).\label{eq:K_G}
\end{equation}

For some word $\epsilon\in\mathcal{A}^{n}$ we define $\overline{\epsilon}:=\epsilon\overline{00}$
the word of infinite length completed from $\epsilon$ with zeros.
Put $\overline{\mathcal{A}^{n}}:=\left\{ \overline{\epsilon};\epsilon\in\mathcal{A}^{n}\right\} $.
One can than define the set of infinite words as $\mathcal{A}^{\infty}:=\cup_{n>0}\overline{\mathcal{A}^{n}}$. 

\begin{figure}

\includegraphics[scale=0.5]{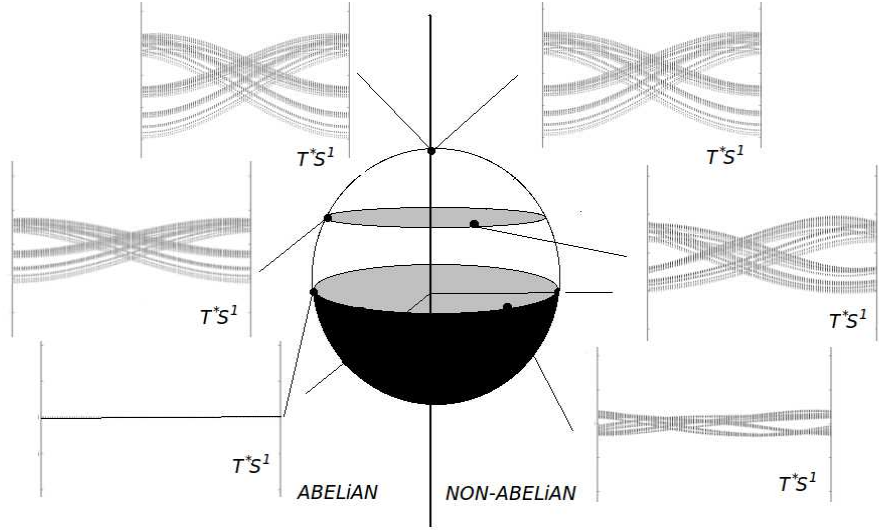}\caption{{\small \label{fig:K_G}Numerical computation of $K_{SU(2)}$ corresponding
to the skew extension of the linear map $E(x)=2x\,\mbox{mod}1$ for
two particular expression of $\tau$. On the left $\tau(x)=e^{i\cos(2\pi x)J_{3}}$
with $J_{3}$ the generator of the rotations around the vertical axis.
In this case $\tau$ maps $S^{1}$ to a $\mbox{U}(1)$ subgroup of
$\mbox{SU}(2)$. The induced dynamics on the sphere is not transitive
as it leaves invariant the geodesics parallel to the equator. Over
any of these, the trapped set corresponds to $K_{U(1)}$ for $\tau(x)=\frac{1}{2\pi}\mathbf{n}_{3}\cos(2\pi x)$,
and degenerates above the equator (this is consistent with the fact
that the canonical map ought not be partially captive in this case).
On the right we break the degeneracy by taking $\tau(x)=e^{i\cos(2\pi x)J_{3}+i0.2\cos(2\pi x)J_{1}}$. }}

\end{figure}

\begin{lem}
\textup{\emph{\label{lem:carat_de_K}$K_{\mathbb{G}}$ can be seen
as the closure of the union of graphs of smooth uniformly bounded
functions $\mathsf{S}_{\epsilon}^{\mathbb{G}};\epsilon\in\mathcal{A}^{\infty}$
over, resp. $S^{1}$ and $S^{1}\times S^{2}$:
\begin{equation}
K_{\mathbb{G}}=\overline{\cup_{\epsilon\in\mathcal{A}^{\infty}}\mathcal{G}\mathsf{S}_{\epsilon}^{\mathbb{G}}.}\label{eq:K_G_as_union}
\end{equation}
 Furthermore, if $S_{\max}^{\mathbb{G}}$ is the uniform bound of
the sequence $\left\{ |\mathsf{S}_{\epsilon}^{\mathbb{G}}|\right\} $
than any point in $F^{-n}\left(Z_{\mathbb{G}}\right)$ is at distance
at most $\left(R+S_{\max}^{\mathbb{G}}\right)e_{\min}^{-n}$ from
$K_{\mathbb{G}}$.}}\end{lem}
\begin{proof}
Points in $F^{-n}\left(Z_{U(1)}\right)$ can be written as 
\[
\bigcap_{x\in S^{1}}\bigcap_{\epsilon\in\mathcal{A}^{n}}\left\{ F^{-n}(x_{\epsilon},\xi);\,\xi\in[-R;R]\right\} .
\]
Using (\ref{eq:F^-n_abel}) a short calculation gives $F^{-n}(x_{\epsilon},\xi)=\left(x,E_{\epsilon}'(x)^{-1}\cdot\xi+S_{\epsilon}^{U(1)}(x)\right)$
with $S_{\epsilon}^{U(1)}$ as in (\ref{eq:S_U(1)}). As $n$ grows
we have that $\left|E_{\epsilon}'(x)^{-1}\cdot\xi\right|\leq e_{\min}^{-n}R\rightarrow0$.
On the other hand, for any $\epsilon\in\mathcal{A}^{\infty}$, any
$n>0$, $\left|S_{\epsilon|_{n}}^{U(1)}(x)\right|\leq\left\Vert \Omega'\right\Vert _{\infty}\frac{1}{e_{\min}-1}=:S_{max}^{U(1)}$
and $\left|S_{\epsilon|_{n}}^{U(1)}(x)-S_{\epsilon|_{n-1}}^{U(1)}(x)\right|=\left|E_{\epsilon|_{n}}'(x)^{-1}\cdot\Omega'\left(x_{\epsilon|_{n}}\right)\right|\leq e_{\min}^{-n}\left\Vert \Omega'\right\Vert _{\infty}.$
The same kind of estimates hold true for the sequences $\partial_{x}^{\alpha}S_{\epsilon|_{n}}^{U(1)}(x)$
for any order of derivation proving the convergence in $C^{\infty}\left(S^{1}\right)$
of the sequence of functions $\left(S_{\epsilon|_{n}}^{U(1)}(x)\right)_{n\geq0}$.
Let us write $\mathsf{S}_{\epsilon}^{U(1)}=\lim_{n\rightarrow\infty}S_{\epsilon|_{n}}^{U(1)}$
and $\mathcal{G}\mathsf{S}_{\epsilon}^{U(1)}:=\left\{ (x,\mathsf{S}_{\epsilon}^{U(1)}(x))\,;x\in I\right\} $
its graph over $S^{1}$. As $n$ grows $F^{-n}\left(Z_{U(1)}\right)$
converges to $\cup_{\bar{\epsilon}\in\overline{\mathcal{A}^{n}}}\mathcal{G}\mathsf{S}_{\overline{\epsilon}}^{U(1)}$
and this gives (\ref{eq:K_G_as_union}) for $\mathbb{G}\equiv\mbox{U}(1)$.
Since, for any $\epsilon\in\mathcal{A}^{\infty}$, 
\[
\left|E_{\epsilon|_{n}}'(x)^{-1}\cdot\xi+S_{\epsilon|_{n}}^{U(1)}(x)-\mathsf{S}_{\epsilon}^{U(1)}(x)\right|\leq e_{\min}^{-n}\left(R+S_{max}^{U(1)}\right),
\]
we get that points in $F^{-n}\left(Z_{U(1)}\right)$ lie at distance
at most $e_{\min}^{-n}\left(R+S_{max}^{U(1)}\right)$ from $K_{U(1)}$.
The exact same argument can be carried out for $\mathbb{G}\equiv\mbox{SU}(2)$
with $S_{\epsilon|_{n}}^{SU(2)}$ as in (\ref{eq:S_SU(2)}) which
is uniformly bounded by $\left\Vert H\right\Vert _{\infty}\frac{1}{e_{\min}-1}=:S_{max}^{SU(2)}$,
and converges in the $C^{\infty}-$topology to $\mathsf{S}_{\epsilon}^{SU(2)}$.
\end{proof}
We do not know much more at present about these elusive sets \cite{Tsujii_00}.
Using the characterization (\ref{eq:K_G_as_union}) one can show quite
easily that over some point $\left\{ x\right\} $ or $\left\{ (x,\mathbf{n})\right\} $
the trapped set is either reduced to a unique point $\{\xi\}$ or
has no isolated points. This however does not inform us about their
dimension more than the obvious bound $1\leq\mbox{dim}K_{U(1)}\leq2$
and $3\leq\mbox{dim}K_{SU(2)}\leq4$.

\subsection{The partially captive property\label{sub:The-partially-captive}}

For any starting point on $K_{\mathbb{G}}$ , for any time $n>0$,
at least one of $k^{n}$ different trajectories Stays in $K_{\mathbb{G}}$.
In the light of lemma \ref{lem:carat_de_K} we have obvious trapped
trajectories from the fact that $\forall\epsilon\in\mathcal{A}^{\infty},\epsilon_{0}\in\mathcal{A},$
\[
F_{\epsilon_{0}}\left(x,\mathsf{S}_{\epsilon_{0}\epsilon}^{U(1)}(x)\right)=\left(x_{\epsilon_{0}},\mathsf{S}_{\epsilon}^{U(1)}(x_{\epsilon_{0}})\right)
\]
 and similarly 
\[
F_{\epsilon_{0}}\left(x,\mathsf{S}_{\epsilon_{0}\epsilon}^{SU(2)}(x,\mathbf{n}),\mathbf{n}\right)=\left(x_{\epsilon_{0}},\mathsf{S}_{\epsilon}^{SU(2)}(x_{\epsilon_{0}},\mathbf{n}_{x,\epsilon_{0}}),\mathbf{n}_{x,\epsilon_{0}}\right),
\]
showing that points in $K_{\mathbb{G}}$ lying on the branch $\mathcal{G}\mathsf{S}_{\epsilon_{0}\epsilon}^{\mathbb{G}}$
jump to the branch $\mathcal{G}\mathsf{S}_{\epsilon}^{\mathbb{G}}$
and so on. When $\tau$ is not a co-boundary, its seems only an unlikely
coincidence that some other trajectory $\epsilon'\neq\epsilon_{0}$
would yield $\mathsf{S}_{\epsilon_{0}\epsilon}^{U(1)}(x)=\mathsf{S}_{\tilde{\epsilon}}^{U(1)}(x_{\epsilon'})$
or $\mathsf{S}_{\epsilon_{0}\epsilon}^{SU(2)}(x,\mathbf{n})=\mathsf{S}_{\tilde{\epsilon}}^{SU(2)}(x_{\epsilon'},\mathbf{n}_{x,\epsilon'})$,
for some $\tilde{\epsilon}\in\mathcal{A}^{\infty}$. Thus one expects
most trajectories to eventually escape in the non-compact direction. 
\begin{defn}
\label{partialy-captive}Let $\mathcal{N}(n)\leq k^{n}$ be the maximal
cardinality (over different starting points in $Z_{\mathbb{G}}$)
of the set of trajectories $\epsilon\in\mathcal{A}^{n}$ that \emph{do
not escape} from $Z_{\mathbb{G}}$. The map $\widehat{E}_{\tau}$
of eq.(\ref{eq:extension}) will be called \emph{partially captive}
iff:
\begin{equation}
\lim_{n\rightarrow\infty}\frac{\log\left(\mathcal{N}(n)\right)}{n}=0.\label{eq:partially_captive}
\end{equation}
 
\end{defn}
This is the hypothesis under which Theorem \ref{thm:gap} holds. It
is known to be generically true in the Abelian setting \cite{Tsujii_07},
but one can reasonably expect the arguments of the quoted article
to remain valid in the non-Abelian case.

\section{Proof of theorem \ref{thm:gap}}

The Abelian case is treated in detail in \cite{fred_gap_09}. We simply
adapt Faure's proof to this context. In the following $1<\kappa<e_{\min}$
and $R>0$ are chosen as in lemma \ref{lem:fuite,}. Using the quantization
rule (\ref{eq:full_quantization}), for any $j>0$, $m<0$ consider
the following Hilbert spaces of distributions 
\[
H_{j^{-1}}^{m}\left(S^{1}\right)\otimes\mathcal{D}_{j}=\mbox{Op}_{j}\left(A_{m}\right)^{-1}\left(L^{2}\otimes\mathcal{D}_{j}\right),
\]
where $A_{m}(x,\xi,\mathbf{n})\equiv A_{m}(\left|\xi\right|)\in(0,1]$
is an elliptic symbol in $S_{0}^{m}\left(T^{*}S^{1}\times S^{2}\right)$
(independent of $x,\mathbf{n}$), constant and equal to one for $|\xi|\leq R$
and equal to $R^{|m|}|\xi|^{m}$ pour $|\xi|\geq R+\eta$ with $\eta>0$
arbitrary small. As subspaces of $\mathcal{D}'(S^{1})$, $H_{j^{-1}}^{m}\left(S^{1}\right)$
and $H^{m}\left(S^{1}\right)$ are isomorphic to one-an-other but
posses different norms. Since the spectrum does not depend on the
choice of a norm, the spectrum of $\hat{F}_{j}:H_{j^{-1}}^{m}\left(S^{1}\right)\otimes\mathcal{D}_{j}\rightarrow H_{j^{-1}}^{m}\left(S^{1}\right)\otimes\mathcal{D}_{j}$
is no other than the Ruelle spectrum of resonances introduced in theorem
\ref{thm:discrete_spect}. Consider 
\[
\hat{Q}_{m}:=\mbox{Op}_{j}\left(A_{m}\right)\hat{F}_{j}\mbox{Op}_{j}\left(A_{m}\right)^{-1}.
\]
By construction $\hat{Q}_{m}$ acts in $L^{2}\left(S^{1}\right)\otimes\mathcal{D}_{j}$
and is unitary equivalent to $\hat{F}_{j}|_{H_{j^{-1}}^{m}\otimes\mathcal{D}_{j}}$.
On the other hand \cite{Kato} $\forall n\in\mathbb{N}^{*}$
\[
r_{s}\left(\hat{Q}_{m}\right)\leq\left\Vert \hat{Q}_{m}^{n}\right\Vert ^{\frac{1}{n}}\leq\left\Vert \hat{Q}_{m}^{n*}\hat{Q}_{m}^{n}\right\Vert ^{\frac{1}{2n}}.
\]
 Define 
\[
\hat{P}^{(n)}:=\hat{Q}^{n*}\hat{Q}^{n}=\mbox{Op}_{j}\left(A_{m}\right)^{-1}\hat{F}_{j}^{n*}\mbox{Op}_{j}\left(A_{m}^{n}\right)\hat{F}_{j}^{n}\mbox{Op}_{j}\left(A_{m}\right)^{-1}.
\]
From Egorov (\ref{eq:egorov_na}) and composition theorems, using
the notations of subsection \ref{sub:Time-n-dynamics}, we get that
$\hat{P}^{(n)}\in OPS_{0}^{m}$ and its symbol reads
\[
P^{(n)}\left(x,\xi;\mathbf{n}\right)=\sum_{\epsilon\in\mathcal{A}^{n}}\frac{1}{E'_{\epsilon}(x)}\frac{A_{m}^{2}(\xi_{x,\mathbf{n},\epsilon})}{A_{m}^{2}(\xi)}\;\mbox{ mod}j^{-1}S_{0}^{-1}.
\]
Faure's simple idea is to use the basic properties of the classical
dynamics to bound this positive symbol, and then use the $L^{2}-$continuity
theorem (lemma \ref{lem:L^2continuity}) to conclude. At $x\in S^{1}$
and $\mathbf{n}\in S^{2}$ fixed we distinguish three cases. Using
lemma \ref{lem:fuite,} we get:
\begin{enumerate}
\item If $|\xi|>R$, then $\forall\epsilon\in\mathcal{A}^{n}$, $\frac{A_{m}^{2}(\xi_{x,\mathbf{n},\epsilon})}{A_{m}^{2}(\xi)}\leq\left(\kappa^{2m}\right)^{n}$.
\item Si $|\xi|\leq R$ but $|\xi_{x,\mathbf{n},\epsilon|_{n-1}}|>R$ then
we can write
\[
\frac{A_{m}^{2}(\xi_{x,\mathbf{n},\epsilon})}{A_{m}^{2}(\xi)}=\overbrace{\frac{A_{m}^{2}(\xi_{x,\mathbf{n},\epsilon})}{A_{m}^{2}(\xi_{x,\mathbf{n},\epsilon|_{n-1}})}}^{\leq\kappa^{2m}}\overbrace{\frac{A_{m}^{2}(\xi_{x,\mathbf{n},\epsilon|_{n-1}})}{A_{m}^{2}(\xi_{x,\mathbf{n},\epsilon|_{n-2}})}...\frac{A_{m}^{2}(\xi_{x,\mathbf{n},\epsilon_{1}})}{A_{m}^{2}(\xi)}}^{\leq1}\leq\kappa^{2m}.
\]

\item In all other case ( $|\xi|\leq R|$ and $|\xi_{x,\mathbf{n},\epsilon|_{n-1}}|\leq R$),
$\frac{A_{m}^{2}(\xi_{x,\mathbf{n},\epsilon})}{A_{m}^{2}(\xi)}\leq1$,
but by definition \ref{partialy-captive}, the number of such trajectories
is bounded by $\mathcal{N}(n-1)$.
\end{enumerate}
Using this decomposition we get
\[
P^{(n)}\left(x,\xi;\mathbf{n}\right)\leq\frac{1}{E_{\min}^{n}}\left(\left(k^{n}-\mathcal{N}(n-1)\right)\kappa^{2m}+\mathcal{N}(n-1)\right)+\mathcal{O}_{n}\left(j^{-1}\right)
\]
 Set $\mathcal{B}(n):=\left(\frac{k}{E_{\min}}\right)^{n}\kappa^{2m}+\frac{\mathcal{N}(n-1)}{E_{\min}^{n}}$.
Remark that at $n$ fixed the first term goes to zero as $m\rightarrow-\infty$.
The $L^{2}-$continuity theorem gives 
\[
\left\Vert \hat{P}^{(n)}\right\Vert \leq\mathcal{B}(n)+\mathcal{O}_{n}\left(j^{-1}\right),
\]
so 
\[
r_{s}\left(\hat{Q}_{m}\right)\leq\left(\mathcal{B}(n)+\mathcal{O}_{n}\left(j^{-1}\right)\right)^{\frac{1}{2n}};\;\forall n\in\mathbb{N}^{*}.
\]
 Letting first $j\rightarrow\infty$, then $m\rightarrow-\infty$
et finally $n\rightarrow\infty$ we finally obtain a nice expression
for the the spectral radius, as $j$ grows: 
\[
r_{s}\left(\hat{F}_{j}|_{H^{m}\otimes\mathcal{D}_{j}}\right)\leq\sqrt{\frac{1}{E_{\min}}\exp\left(\liminf_{n\rightarrow\infty}\left(\frac{\log\left(\mathcal{N}(n)\right)}{n}\right)\right)}+o(1).
\]
 The partially captive assumption (\ref{eq:partially_captive}) then
yields (\ref{eq:gap}) of Theorem \ref{thm:gap}. The second statement
of Theorem \ref{thm:gap} is derived using $\left\Vert \hat{Q}_{m}^{n}\right\Vert \leq\sqrt{\left\Vert \hat{P}^{(n)}\right\Vert }=\left(\mathcal{B}(n)+\mathcal{O}_{n}\left(j^{-1}\right)\right)^{\frac{1}{2}}$
(polar decomposition \cite{gohberg-00}). With the partially captive
assumption, for any $c>0$ an $n$ large enough $\mathcal{N}(n)<e^{nc}$,
so for any $\rho>E_{\min}^{-1/2}$ , $\frac{\mathcal{N}(n)}{E_{\min}^{n}}<\rho^{2n}.$
Thus for $m$ sufficiently negative, $j,\, n$ sufficiently large
$\left\Vert \hat{Q}_{m}^{n}\right\Vert :=\left\Vert \hat{F}^{n}\right\Vert _{H_{j^{-1}}^{m}\otimes\mathcal{D}_{j}}\leq\rho^{n}$.

\section{Proof of theorem \ref{thm:weyl_law}}

To treat simultaneously both cases $\mathbb{G}\equiv\mbox{U}(1)$;
$\mathbb{G}\equiv\mbox{SU}(2)$ let us fix some notations. Let $M_{\mathbb{G}}$
be the classical phase space associated to the FIO $\hat{F}_{\alpha}$;
so $M_{U(1)}\equiv T^{*}S^{1}$ and $M_{SU(2)}\equiv T^{*}S^{1}\times S^{2}$.
We write points on $M_{\mathbb{G}}$ as $\rho=(\rho_{c},\xi)$ with
$\rho_{c}$ the components along the compact directions of $M_{\mathbb{G}}$.
We set $\hbar=\nu^{-1},\,\nu>0$ or $\hbar=j^{-1},\, j>0$ and $\mbox{O}\mbox{p}$
will stand for the associated quantification either (\ref{eq:weyl_quant})
or (\ref{eq:full_quantization}).

\subsection{The escape function}

For any closed subset $A$ of $M_{\mathbb{G}}$ we denote by $A^{\delta}$
its closed $\delta-$neighbourhood.
\begin{lem}
(Existence of an escape function). \label{lem:Escape-function}$\exists\mathsf{C}_{0},\,\mathsf{C}_{1}>0$
such that $\forall m<0,$ $\forall\,0\leq\mu<\frac{1}{2}$; and $\forall\,1<\kappa<e_{\min}$,
there exists an elliptic symbol in $\hbar^{\mu m}S_{\mu}^{m}\left(M_{\mathbb{G}}\right)$,
called the \emph{escape function} and written $A_{m,\mu}^{\mathbb{G}}$,
satisfying the following property:

$\forall\rho\notin K_{\mathbb{G}}^{\mathsf{C}_{0}\hbar^{\mu}}$, $\forall\epsilon\in\mathcal{A}$,
$A_{m,\mu}^{\mathbb{G}}$ deacreases stricly along the trajectories
of $F_{\epsilon}$:
\begin{equation}
\frac{A_{m,\mu}^{\mathbb{G}}\circ F_{\epsilon}}{A_{m,\mu}^{\mathbb{G}}}(\rho)\leq\mathsf{C}_{1}\kappa^{m}.\label{eq:escape_estimate}
\end{equation}
\end{lem}
\begin{proof}
Let $m<0$. Let $R$ be as in lemma \ref{lem:fuite,} and define $A_{m}\in S_{0}^{m}\left(M_{\mathbb{G}}\right)$;
$A_{m}(\rho)\in(0,1]$ s.t.
\begin{equation}
\begin{array}{cccc}
A_{m}(\rho) & = & 1 & \mbox{if }\ensuremath{|\xi|}\leq R\\
 & = & \left(\frac{|\xi|}{R}\right)^{m} & \mbox{if }\ensuremath{|\xi|}\geq R+\eta
\end{array}\label{eq:A_m}
\end{equation}
 with $\eta>0$ arbitrarily small. Put
\begin{equation}
\tilde{A}_{m,\mu}^{\mathbb{G}}:=\frac{1}{k^{n}}\sum_{\epsilon\in\mathcal{A}^{n}}\frac{A_{m}\circ F_{\epsilon}^{n}}{\left|E'_{\epsilon}\right|^{m}};\label{eq:A_m,nu,tilda}
\end{equation}
 with $\mbox{ with }n=n(\hbar,\mu)$ such that
\begin{equation}
e_{\min}^{-n}=\mathcal{O}(1)\hbar^{\mu}\Leftrightarrow n(\hbar,\mu)=[\mu\frac{\log\hbar^{-1}}{\log e_{min}}].\label{eq:n(h,mu)}
\end{equation}
For any point $\rho\notin F^{-n}\left(Z_{\mathbb{G}}\right)$, we
have, by lemma \ref{lem:fuite,} , (\ref{eq:F^n}) and (\ref{eq:A_m}),
that
\[
\tilde{A}_{m,\mu}^{\mathbb{G}}(\rho):=\frac{R^{|m|}}{k^{n}}\sum_{\epsilon\in\mathcal{A}^{n}}\left|\xi-S_{\epsilon}^{\mathbb{G}}(\rho_{c})\right|^{m}.
\]
 Since, from the proof of lemma \ref{lem:carat_de_K}, we know that
$S_{\epsilon}^{\mathbb{G}}(\rho_{c})$ is smooth (uniformly in $n$),
we get directly, with (\ref{eq:n(h,mu)}), that the symbol class estimates
(\ref{eq:S^m_nu}) and (\ref{eq:S^m_mu(T*S_times_S^2)}) of $\hbar^{\mu m}S_{\mu}^{m}$
are satisfied as long as $\left|\xi-S_{\epsilon}^{\mathbb{G}}(\rho_{c})\right|\geq e_{\min}^{-n}=\mathcal{O}(1)\hbar^{\mu}$. 

By lemma \ref{lem:carat_de_K} if $\rho$ is at distance at least
$(R+2S_{\max}^{\mathbb{G}})e_{\min}^{-n}$ from $K_{\mathbb{G}}$
than $\rho$ is both out of $F^{-n}(Z_{\mathbb{G}})$ and satisfies
$\left|\xi-S_{\epsilon}^{\mathbb{G}}(\rho_{c})\right|\geq e_{\min}^{-n}$.
We can always smooth out $\tilde{A}_{m,\mu}^{\mathbb{G}}$ near $K_{\mathbb{G}}$
to define $A_{m,\mu}^{\mathbb{G}}\in\hbar^{\mu m}S_{\mu}^{m}$ so
that $\tilde{A}_{m,\nu}^{\mathbb{G}}\equiv A_{m,\nu}^{\mathbb{G}}$
out of the neighbourhood $K_{\mathbb{G}}^{\mathsf{C}_{0}e_{\min}^{-n}}\supseteq F^{-n}(Z_{\mathbb{G}})$
with $\mathsf{C}_{0}=2(R+2S_{\max}^{\mathbb{G}})$. 

On the other hand, again by lemma \ref{lem:fuite,} and (\ref{eq:A_m,nu,tilda}),
out of $F^{-n}(Z_{\mathbb{G}})$ we also have that, for any letter
$\epsilon$, 
\[
\frac{\tilde{A}_{m,\mu}^{\mathbb{G}}\left(F_{\epsilon_{0}}(\rho)\right)}{\tilde{A}_{m,\mu}^{\mathbb{G}}(\rho)}\leq\kappa^{m}\frac{\sum_{\epsilon\in\mathcal{A}^{n}}\left|E'_{\epsilon}(x)\right|^{m}}{\sum_{\epsilon\in\mathcal{A}^{n}}\left|E'_{\epsilon}(x_{\epsilon_{0}})\right|^{m}}.
\]
Now, for some absolute constant $C>0$, $e^{nP(|m|)-C}\leq\sum_{\epsilon\in\mathcal{A}^{n}}\left|E'_{\epsilon}(x)\right|^{m}\leq e^{nP(|m|)+C}$
with $P(|m|)$ the topological pressure of $E$ associated to the
potential $-m\log E'$ (see \cite{Falconer_tech} Theorem 5.1 p. 72).
Thus choosing $\mathsf{C}_{1}=e^{2C}$ concludes the proof of lemma
\ref{lem:Escape-function} with $A_{m,\mu}^{\mathbb{G}}$ as the escape
function
\end{proof}

\subsection{End of the proof}

Since $A_{m,\mu}^{\mathbb{G}}$ is elliptic and of order $m$ the
following spaces of distributions 
\[
\mathcal{H}_{\alpha,\mu}^{m}:=\mbox{Op}\left(A_{m,\mu}^{\mathbb{G}}\right)^{-1}\left(L^{2}\left(S^{1}\right)\otimes\mathcal{D}_{\alpha}\right)
\]
 are Hilbert spaces w.r. to the norm inherited from $L^{2}$ and are
isomorphic in terms of subsets of $\mathcal{D}'\left(S^{1}\right)\otimes\mathcal{D}_{\alpha}$
to $H^{m}\left(S^{1}\right)\otimes\mathcal{D}_{\alpha}$. Theorem
\ref{thm:weyl_law} essentially reduces to the following statement:
\begin{lem}
\label{lem:weyl_Q_m_mu}Choose any $\mathsf{C}>\mathsf{C}_{0}$ and
set $C_{U(1)}^{-1}=2\pi$; $C_{SU(2)}^{-1}=8\pi^{2}$. $\forall\epsilon>0$,
$\forall0\leq\mu<\frac{1}{2}$, $m<0$ sufficiently negative and $|\alpha|>0$
large enough, 
\[
\sharp\left\{ \mbox{spct}\left(\hat{F}_{\alpha}|_{\mathcal{H}_{\alpha,\mu}^{m}}\right)\bigcap\mathbb{C}\backslash D_{\epsilon}^{\mathbb{C}}\right\} \leq C_{\mathbb{G}}\mbox{dim}_{\mathbb{C}}\left(\mathcal{D}_{\alpha}\right)|\alpha|\mbox{Vol}\left\{ K_{\mathbb{G}}^{\mathsf{C}|\alpha|{}^{-\mu}}\right\} \left(1+o(1)\right).
\]
\end{lem}
\begin{proof}
$\hat{F}_{\alpha}:\mathcal{H}_{\alpha,\mu}^{m}\rightarrow\mathcal{H}_{\alpha,\mu}^{m}$
is by construction unitary equivalent to 
\[
\hat{Q}_{m,\mu}:=\mbox{Op}\left(A_{m,\mu}^{\mathbb{G}}\right)\hat{F}_{\alpha}\mbox{Op}\left(A_{m,\mu}^{\mathbb{G}}\right)^{-1}:L^{2}\left(S^{1}\right)\otimes\mathcal{D}_{\alpha}\rightarrow L^{2}\left(S^{1}\right)\otimes\mathcal{D}_{\alpha}.
\]
Define 
\[
\hat{P}_{\mu}:=\hat{Q}_{m,\mu}^{*}\hat{Q}_{m,\mu}=\mbox{Op}\left(A_{m,\mu}^{\mathbb{G}}\right)^{-1}\hat{F}_{\alpha}^{*}\mbox{Op}\left(A_{m,\mu}^{\mathbb{G}}\right)^{2}\hat{F}_{\alpha}\mbox{Op}\left(A_{m,\mu}^{\mathbb{G}}\right)^{-1}.
\]
By the composition and Egorov theorems (\ref{eq:egorov_a}), (\ref{eq:egorov_na})
$\hat{P}_{\mu}\in OPS_{\mu}^{0}$ and its symbol reads
\[
P_{\mu}=\sum_{\epsilon\in\mathcal{A}}\left(\frac{A_{m,\mu}^{\mathbb{G}}\circ F_{\epsilon}}{A_{m,\mu}^{\mathbb{G}}}\right)^{2}\mbox{ mod }\hbar^{1-2\mu}S_{\mu}^{-1}.
\]
 From lemma \ref{lem:Escape-function}, $P_{\mu}$ naturally decomposes
into a compact part $K_{\mu}$ supported on $K_{\mathbb{G}}^{\mathsf{C}\hbar^{\mu}}$
, with some $\mathsf{C}>\mathsf{C}_{0}$, and a bounded part $R_{\mu}$
with $\sup\left|R_{\mu}\right|\leq\mbox{C}_{1}^{2}\kappa^{2m}+\mathcal{O}(\hbar^{1-2\mu})$.
With lemma \ref{lem:L^2continuity}, the decomposition transposes
to the operator level with $\hat{P}_{\mu}=\hat{K}_{\mu}+\hat{R}_{\mu}$,
$\hat{K}_{\mu}:=\mbox{Op}_{\hbar}\left(K_{\mu}\right)$ trace class
and self-adjoint and $\left\Vert \hat{R}_{\mu}\right\Vert \leq\mbox{C}_{1}^{2}\kappa^{2m}+\mathcal{O}(\hbar^{1-2\mu}).$
From lemma \ref{lem:weyl_law_compact} in the appendix , we have that
$\forall\epsilon>0$ and $\hbar\equiv|\alpha|^{-1}$ small enough
\[
\sharp\left\{ \mbox{spct}\left(\hat{K}_{\mu}\right)\bigcap\mathbb{R}\backslash(-\epsilon;\epsilon)\right\} \leq f_{\mathbb{G}}(\alpha)|\alpha|\mbox{Vol}\left\{ K_{\mathbb{G}}^{\mathsf{C}\hbar^{\mu}}\right\} \left(1+o(1)\right),
\]
with $f_{U(1)}(\nu)\equiv\frac{1}{2\pi}$ and $f_{SU(2)}(j)=\frac{1}{8\pi^{2}}\mbox{dim}_{\mathbb{C}}\left(\mathcal{D}_{j}\right)$.
By perturbation, eventually choosing a larger $\epsilon>0$, for $m$
sufficiently negative and $\hbar$ small enough, the same is true
for the eigenvalues of $\hat{P}_{\mu}$ thus for the singular values
of $\hat{Q}_{m,\mu}$. Corollary \ref{cor:singular} from the appendix
allows us to draw the same conclusion for the eigenvalues of $\hat{Q}_{m,\mu}$,
yielding the result. \end{proof}
\begin{defn}
\label{The-upper-Minkowski}The {\small upper Minkowski}\textbf{}%
\footnote{For nice sets the Minkowski dimension coincides with the Hausdorff
dimension $\mbox{dim}_{H}$, but in general $\mbox{dim}_{H}A\leq\dim A$.%
} dimension (or box dimension) of a non empty bounded subset $A$ of
$\mathbb{R}^{d}$ is 
\begin{equation}
d-\dim A:=\mbox{co}\dim A:=\sup_{s\in\mathbb{R}}\left\{ \limsup_{\delta\downarrow0}\delta^{-s}\cdot\mbox{Vol}^{d}\left(A^{\delta}\right)<+\infty\right\} .\label{eq:defdim}
\end{equation}
In general $\limsup_{\delta\downarrow0}\delta^{-co\dim A}\cdot\mbox{Vol}^{d}\left(A^{\delta}\right)<+\infty$
does not hold%
\footnote{when it does the set $A$ is said to be of pure dimension see \cite{sjoestrand_90}
for some comments and further references.%
}, so $\mbox{Vol}^{d}\left(A^{\delta}\right)=\mathcal{O}\left(\delta^{co\dim A-\eta}\right)$
for any $\eta>0$. We write the latter $\mbox{Vol}^{d}\left(A^{\delta}\right)=\mathcal{O}\left(\delta^{co\dim A-0}\right).$ 
\end{defn}
From this definition and lemma \ref{lem:weyl_Q_m_mu}, Theorem \ref{thm:weyl_law}
follows rather directly. For any $0\leq\mu<\frac{1}{2}$, the spectrum
of $\hat{F}_{\alpha}|_{\mathcal{H}_{\alpha,\mu}^{m}}$ is no other
than the Ruelle spectrum of resonances, the latter independent of
$\mu$. Thus, for any $\epsilon>0$, $m<0$ sufficiently negative,
for some $\tilde{C}_{\mathbb{G}}>0$ independent of $\alpha,\, m$,
and for $|\alpha|$ large enough: 
\begin{equation}
\sharp\left\{ \mbox{spct}\left(\hat{F}_{\alpha}|_{H^{m}\otimes\mathcal{D}_{\alpha}}\right)\bigcap\mathbb{C}\backslash D_{\epsilon}^{\mathbb{C}}\right\} \leq\tilde{C}_{\mathbb{G}}\mbox{dim}_{\mathbb{C}}\left(\mathcal{D}_{\alpha}\right)|\alpha|^{1-\frac{1}{2}co\dim K_{\mathbb{G}}+0}\label{eq:derniere_etape}
\end{equation}
In the Abelian case 
\[
\mbox{dim}_{\mathbb{C}}\left(\mathcal{D}_{\alpha}\right)|\alpha|^{1-\frac{1}{2}co\dim K_{\mathbb{G}}+0}=\left|\nu\right|^{\frac{1}{2}\left(2-(2-\dim K_{\mathbb{G}})\right)+0}=\left|\nu\right|^{\frac{1}{2}\dim K_{\mathbb{G}}+0}.
\]
In the non-Abelian case, 
\[
\mbox{dim}_{\mathbb{C}}\left(\mathcal{D}_{\alpha}\right)|\alpha|^{1-\frac{1}{2}co\dim K_{\mathbb{G}}+0}=\left(\frac{2j+1}{j}\right)j^{\frac{1}{2}\dim K_{\mathbb{G}}+0}.
\]
Thus (\ref{eq:derniere_etape}) yields Theorem \ref{thm:weyl_law}.

\section{Appendix A. Adapted Weyl type estimates}

If $a\in S_{\mu}^{0}\cap L^{2}\left(\mathbb{R}^{2d}\right)$ one has
the following important \emph{exact} formula \cite{zworski-03}: 
\begin{equation}
\mbox{tr}\left(\mbox{Op}_{\hbar}^{w}(a)\right)=\frac{1}{\left(2\pi\hbar\right)^{d}}\int a(x,\xi)dxd\xi.\label{eq:tracePDO}
\end{equation}
 For the quantization $\mbox{Op}_{j}:=\mbox{Op}_{j}^{AW}\circ\mbox{Op}_{j^{-1}}^{w}$
defined in (\ref{eq:full_quantization}), using (\ref{eq:anti-wick})
and (\ref{eq:tracePDO}) we have for any $a\in S_{\mu}^{0}(T^{*}S^{1}\times S^{2})\cap L^{2}\left(T^{*}S^{1}\times S^{2}\right)$:
\begin{equation}
\mbox{tr}\left(\mbox{Op}_{j}(a)\right)=\frac{\dim_{\mathbb{C}}\left(\mathcal{D}_{j}\right)j}{8\pi^{2}}\int_{T^{*}S^{1}\times S^{2}}a(x,\xi,\mathbf{n})dxd\xi d\mathbf{n}\label{eq:trace_full_quant}
\end{equation}

\begin{lem}
\label{lem:weyl_law_compact}Let $a\in S_{\mu}^{-\infty}$ be a real
compacly supported symbol. $\forall\hbar>0,$ $\mbox{\emph{Op}}_{\hbar}^{w}(a)$
is self-adjoint and trace class on $L^{2}$. Furthermore, for any
$\epsilon>0$, $\hbar$ small enough: 
\begin{equation}
\sharp\left\{ \mbox{\emph{spct}}\left(\mbox{\emph{Op}}_{\hbar}^{w}(a)\right)\bigcap\mathbb{R}\backslash(-\epsilon;\epsilon)\right\} =\frac{1}{\left(2\pi\hbar\right)^{d}}(\,\mbox{\emph{Vol}}\left\{ |a|>\epsilon\right\} +\mbox{\emph{Vol}}\left\{ |a|>0\right\} o(1)\,)\label{eq:loi_weyl}
\end{equation}
 The same holds true for $a\in S_{\mu}^{-\infty}\cap C_{0}^{\infty}\left(T^{*}S^{1}\times S^{2}\right)$,
for $j>0$ large enough, with 
\begin{equation}
\sharp\left\{ \mbox{\emph{spct}}\left(\mbox{\emph{Op}}_{j}(a)\right)\bigcap\mathbb{R}\backslash(-\epsilon;\epsilon)\right\} =\frac{\dim_{\mathbb{C}}\left(\mathcal{D}_{j}\right)j}{8\pi^{2}}(\,\mbox{\emph{Vol}}\left\{ |a|>\epsilon\right\} +\mbox{\emph{Vol}}\left\{ |a|>0\right\} o(1)\,)\label{eq:loi_weyl_Op_j}
\end{equation}
\end{lem}
\begin{proof}
Consider $1_{\epsilon}$ $-$the characteristic function of $\mathbb{R}\backslash(-\epsilon;\epsilon)-$
so that 
\[
\mbox{tr}1_{\epsilon}\left(\hat{A}\right)=\sharp\left\{ \mbox{spct}\left(\hat{A}\right)\bigcap\mathbb{R}\backslash(-\epsilon;\epsilon)\right\} .
\]
Let $P_{\pm}$ be polynomials of degree $N$ vanishing at zero and
approximating $1_{\epsilon}$ on $\left[-C;C\right]$, with $C>\sup|a|$,
s.t. $P_{-}(t)\leq1_{\epsilon}(t)\leq P_{+}(t)$ for any $t\in\left[-C;C\right]$.
Let us write $\hat{A}$ for $\mbox{Op}_{\hbar}^{w}(a)$. $P_{\pm}\left(\hat{A}\right)$
are well defined PDOs in $OPS_{\mu}^{-\infty}$ and, by the composition
theorem, their respective symbols read $P_{\pm}(a)+\hbar^{2(1-2\mu)}b_{\pm}$
with $b_{\pm}$ negligible out of the support of $a$. By eq.(\ref{eq:tracePDO})
:
\[
\mbox{tr}P_{\pm}\left(\hat{A}\right)=\frac{1}{\left(2\pi\hbar\right)^{d}}\int\left(P_{\pm}(a)+\hbar^{2(1-2\mu)}b_{\pm}\right)dxd\xi.
\]
 Now $\left|\int b_{\pm}dxd\xi\right|\leq\mbox{Vol}\left\{ |a|>0\right\} C_{N}$
with $C_{N}$ independent of $\hbar$. On the other hand $P_{\pm}=1_{\epsilon}+r_{\pm}$
on $[-C;C]$ with $r_{\pm}(0)=0$. We thus get $\int P_{+}(a)dxd\xi\leq\mbox{Vol}\left\{ |a|>\epsilon\right\} +\mbox{Vol}\left\{ |a|>0\right\} \delta_{N}$
and the opposite inequality for $\int P_{-}(a)dxd\xi$, with $\delta_{N}\rightarrow0$
as $N$ grows. By the spectral and $L^{2}-$continuity (lemma \ref{lem:L^2continuity})
theorems, for $\hbar>0$ small enough, 
\[
\mbox{tr}P_{-}\left(\hat{A}\right)\leq\mbox{tr}1_{\epsilon}\left(\hat{A}\right)\leq\mbox{tr}P_{+}\left(\hat{A}\right)
\]
 so
\[
\frac{1}{\left(2\pi\hbar\right)^{d}}(\,\mbox{Vol}\left\{ |a|>\epsilon\right\} -\mbox{Vol}\left\{ |a|>0\right\} (\delta_{N}+\hbar^{2(1-2\mu)}C_{N})\,)\leq\mbox{tr}1_{\epsilon}\left(\hat{A}\right)
\]
 and
\[
\mbox{tr}1_{\epsilon}\left(\hat{A}\right)\leq\frac{1}{\left(2\pi\hbar\right)^{d}}(\,\mbox{Vol}\left\{ |a|>\epsilon\right\} +\mbox{Vol}\left\{ |a|>0\right\} (\delta_{N}+\hbar^{2(1-2\mu)}C_{N})\,).
\]
As long as $\mu<\frac{1}{2}$, for any $\delta>0$ arbitrarily small,
one can take $N$ large enough s.t. $\delta_{N}\leq\delta/2$ and
then $\hbar$ small enough s.t. $\hbar^{2(1-2\mu)}C_{N}\leq\delta/2$,
so that the term $\delta_{N}+\hbar^{2(1-2\mu)}C_{N}$ is smaller than
$\delta$. This gives (\ref{eq:loi_weyl}). The exact same argument
can be carried out for $\hat{A}:=\mbox{Op}_{j}(a)$, using (\ref{eq:trace_full_quant})
to compute the trace of $P_{\pm}\left(\hat{A}\right)$ and with $\hbar\equiv j^{-1}$
to get (\ref{eq:loi_weyl_Op_j}).
\end{proof}

\section{Appendix B. General lemmas on singular values}

Let $(P_{\nu})_{\nu\in\mathbb{N}}$ be a family of compact operators
on some Hilbert space. Consider any $P_{\nu}$ and let $(\lambda_{j,\nu})_{j\in\mathbb{N}^{*}}\in\mathbb{C}$
be the sequence of its eigenvalues ordered decreasingly according
to multiplicity:
\[
|\lambda_{1,\nu}|\geq|\lambda_{2,\nu}|\geq...
\]
In the same manner, define $(\mu_{j,\nu})_{j\in\mathbb{N}^{*}}\in\mathbb{R}^{+}$,
the decreasing sequence of singular values of $P_{\nu}$ ( the eigenvalues
of $\sqrt{P_{\nu}^{*}P_{\nu}}$ ). Finally let $[x]\in\mathbb{N}$
stand for the integral part of $x\in\mathbb{R}$. 
\begin{lem}
\label{lem:``Singular-and-eigenvalues} Suppose there exits a map
$N:\mathbb{N}\rightarrow\mathbb{N}$ s.t. $N(\nu)\rightarrow\infty$
and \textup{$\mu_{N(\nu),\nu}\rightarrow0$ as $\nu$ grows.} Then
$\forall C>1,$ $|\lambda_{[C.N(\nu)],\nu}|\rightarrow_{\nu\rightarrow\infty}0$. \end{lem}
\begin{cor}
\label{cor:singular} Let $N\,:\,\mathbb{N}\rightarrow\mathbb{N}$
be as in lemma \ref{lem:``Singular-and-eigenvalues}. Suppose that
$\forall\epsilon>0,$ $\exists A_{\epsilon}\geq0$ s.t. $\forall\nu\geq A_{\epsilon};\;\#\left\{ \, j\in\mathbb{N}^{*}\;|\;\mu_{j,\nu}>\epsilon\right\} <N(\nu).$
Then for any $C>1,$ $\epsilon>0$ there exists $B_{C,\epsilon}\geq0$
such that: 
\begin{equation}
\forall\nu\geq B_{C,\epsilon};\;\#\left\{ \, j\in\mathbb{N}^{*}\;|\;|\lambda_{j,\nu}|>\epsilon\right\} \leq[C\times N(\nu)].\label{eq:cor2}
\end{equation}
\end{cor}
\begin{proof}
(Of corollary \ref{cor:singular}). Suppose that for any $\epsilon>0$,
there exists a rank $A_{\epsilon}$ s.t. for all $\nu\geq A_{\epsilon}$
$\#\left\{ j\in\mathbb{N}^{*}\;;\;\mu_{j,\nu}\right\} <N(\nu),$ which
means that $\mu_{N(\nu),\nu}\rightarrow_{\nu\rightarrow\infty}0$
and from Lemma \ref{lem:``Singular-and-eigenvalues}, $\forall C>1$,
$|\lambda_{[CN(\nu)],\nu}|\rightarrow_{\nu\rightarrow\infty}0,$ which
can be directly restated as (\ref{eq:cor2}).
\end{proof}
Let us now prove the lemma.
\begin{proof}
(Of lemma \ref{lem:``Singular-and-eigenvalues}) The main relation
between singular and eigenvalues is given by the Weyl inequalities
( see \cite{gohberg-00} p. 50 for a proof):
\begin{equation}
\prod_{j=1}^{k}\mu_{j,\nu}\leq\prod_{j=1}^{k}|\lambda_{j,\nu}|;\;\;\;\;\forall k\in\mathbb{N}^{*}.\label{eq:weyl_inequalities}
\end{equation}
Let $m_{j,\nu}:=-log\left(\mu_{j,\nu}\right)$, $l_{j,\nu}:=-log\left(|\lambda_{j,\nu}|\right)$
to define $S_{k,\nu}:=\sum_{j=1}^{k}m_{j,\nu},$ and $L_{k,\nu}:=\sum_{j=1}^{k}l_{j,\nu}$.
The Weyl inequalities (\ref{eq:weyl_inequalities}) thus reads: $S_{k,\nu}\leq L_{k,\nu}$,
$\forall k\in\mathbb{N}^{*}$. Notice that both sequences $\left(l_{j,\nu}\right)_{j\geq1}$
and $\left(m_{j,\nu}\right)_{j\geq1}$ are increasing so, $\forall k\in\mathbb{N}^{*}$,
$k\cdot l_{k,\nu}\geq L_{k,\nu}$, and for any $k,K\in\mathbb{N}^{*},$
\begin{equation}
S_{k+K,\nu}\geq K\cdot m_{k,\nu}.\label{eq:somme_sing}
\end{equation}
Suppose that $\mu_{N(\nu),\nu}\rightarrow0$ (hence $m_{N(\nu),\nu}\rightarrow\infty$)
as $\nu\rightarrow\infty$ and choose some constant $C>1$, By (\ref{eq:somme_sing})
we have that 
\begin{equation}
S_{[CN(\nu)],\nu}\geq\left([CN(\nu)]-N(\nu)\right)\cdot m_{N(\nu),\nu},\label{eq:S_vs_m}
\end{equation}
and therefore, since $l_{[CN(\nu)],\nu}\geq\frac{1}{[CN(\nu)]}\times L_{[CN(\nu)],\nu}\geq\frac{1}{[CN(\nu)]}\times S_{[CN(\nu)],\nu}$,
from (\ref{eq:S_vs_m}) we get
\begin{equation}
l_{[CN(\nu)],\nu}\geq\frac{[CN(\nu)]-N(\nu)}{[CN(\nu)]}\times m_{[CN(\nu)],\nu}.\label{eq:l_CN}
\end{equation}
Notice that $[CN(\nu)]-N(\nu)>0$ for $\nu$ large enough. Therefore
(\ref{eq:l_CN}) gives the result.
\end{proof}
\bibliographystyle{plain}
\bibliography{articles}

\end{document}